\documentclass[14pt]{amsart}
\usepackage{amsfonts}
\usepackage{amsfonts,latexsym,rawfonts,amsmath,amssymb,amsthm}
\usepackage[plainpages=false]{hyperref}
\usepackage{enumerate}

\usepackage{graphicx}

\RequirePackage{color}

\numberwithin{equation}{section}
\renewcommand{\theequation}{\arabic{section}.\arabic{equation}}

\def\L{\mathcal L}
\def\M{\mathcal M}

\def\R{\Bbb R}

\def\I{\mathcal I}
\def\D{\mathcal D}

\def\P{\mathcal P}

\def\K{\mathcal K}

\def\Q{\mathcal Q}

\def\T{\mathcal T}

\newtheorem{Proposition}{Proposition}[section]
\newtheorem{Theorem}[Proposition]{Theorem}
\newtheorem{Lemma}[Proposition]{Lemma}

\newtheorem{Corollary}[Proposition]{Corollary}
\newtheorem{Remark}[Proposition]{Remark}
\newtheorem{Example}[Proposition]{Example}
\newtheorem{Definition}[Proposition]{Definition}

\title{Strong maximum principle for  generalized solutions to   equations of the Monge-Amp\`ere type}
\thanks{This work was supported by NSFC   12141103.}
\begin{document}

	%	\address{Department of Mathematical Sciences, Tsinghua University}
	%	\email{\tt hjian@tsinghua.edu.cn}
	%	
	%	\address{Department of Mathematics, The Hong Kong University of Science and Technology}
	%	\email{\tt maxstu@ust.hk  }
	
	%\date{}
	
	\bibliographystyle{plain}
	
	%
	%\tableofcontents
	\maketitle

	\baselineskip=15.8pt
	\parskip=3pt
	%%%submitted to JDG on December 2011

	\centerline { Huaiyu Jian\ \  and \ \ Xushan Tu}
	
	%\centerline {\bf  Huaiyu Jian}
	%\centerline {Department of Mathematical sciences, Tsinghua University}
	%\centerline {Beijing 100084, China}
	%
	%
	%\vskip10pt
	%
	% \centerline { \bf Xushan Tu}
	%\centerline {Department of Mathematics, The Hong Kong University of Science and Technology}
	%\centerline {Clear Water Bay, Kowloon, Hong Kong}
	%

	\vskip 15pt
	
	\begin{abstract}		
		
		In this paper, we investigate the strong maximum principle for generalized solutions of Monge-Amp\`ere type equations. We prove that the strong maximum principle holds at points where the function is strictly convex but not necessarily $C^{1,1}$ smooth, and show that it fails at non-strictly convex points. The results we obtain can be applied to various Minkowski type problems in convex geometry by the virtue  of the Gauss image map.
		%		and give the uniqueness of solutions for special cases.  This is the first study of the strong comparison principle for non-smooth solutions to Monge-Amp\`ere equations.
	\end{abstract}
	
	\vskip 15pt
	

	\vskip 15pt
	
	
	\vskip20pt

	\baselineskip=15.8pt
	\parskip=3pt

	%	\centerline {\bf  Strong maximum  principle for generalized solutions  to the Monge-Amp\`ere type equations      }
	%	

	\tableofcontents
	%\maketitle

	\baselineskip=15.8pt
	\parskip=3.0pt

	\newpage
	\section{Introduction}

	The Strong Maximum Principle by Hopf \cite{[HE]} is a fundamental theorem in the study of partial differential equations.  The most general version can be found in Gilbarg and Trudinger's monograph \cite[Theorem 3.5]{[GT]}, which states that any non-constant $C^2$ subsolution $u$ of a uniformly elliptic equation $a^{ij} D_{ij} u +b_iD_iu = 0 $ does not attain its maximum value in the interior.
	
	Naturally, we pose a similar question for the Monge-Amp\`ere equation. Let $u $ and $v$ be two convex solutions of
	\begin{equation}\label{eq:ma measure f}
		\det D^2 u=\det D^2 v =f(x) ,  \quad 0 < \lambda \leq   f(x)  \leq \Lambda.
	\end{equation}
Throughout this paper, we assume that  $\lambda $  and $ \Lambda$ are two positive constants.
	 Suppose that $u-v$ attains a local maximum or minimum, is it true that $u-v$ is a constant?  Note that
	\[  \int_{0}^{1}A^{ij,t}dt  \cdot D_{ij}(u-v)=\det D^2 u -\det D^2 v =0,\]
	where $A^{ij,t}$ are co-factors of $tD^2u+(1-t)D^2v$ for each $t$. When $u$ and $v$ are both strict convex, and $ f \in C_{loc}^{\alpha}$,   Caffarelli  \cite{[C1],[C2]} obtained the Schauder regularity of $u$ and $v$; see \cite{[JW07]} for a simplified proof. In this case, the corresponding linearized operator is uniformly elliptic, and the classical strong maximum principle implies that $u - v$ is a constant. On the other hand, Wang \cite{[W95]} constructed a class of strict convex functions  $u\notin C^{1,1}$, whose $\det D^2 u$ are   bounded and positive. Therefore, the linearized operator is generally not uniformly elliptic, and the strong maximum principle for the generalized solutions of the Monge-Amp\`ere equation remains open.
	
	This paper aims to study the strong maximum principle for generalized convex solutions $u$ (See Definition \ref{def:solution 01}) of Monge-Amp\`ere type equations.  Our focus of study is the set of non-strictly convex points of $u$. We say that {\sl a point  $x_0$ is a strictly convex point of $u$ if
		\begin{equation}\label{eq:def:strict convex}
			u\left( \frac{x+x_0}{2}\right) <  \frac{u(x)+u(x_0)}{2} \text{ for all } x \neq x_0,
		\end{equation}
		and let  $\Sigma_u $ as the set of non-strictly convex points.  }	
	Assuming that $0<\lambda <\det D^2 u <\Lambda$ in $\Omega$,  then $\Sigma_u  $ is a closed set and $u\in C_{loc}^{1,\alpha}(\Omega \setminus \Sigma_u )$. In this case, $\Sigma_u$ is the union of all such convex sets $E $  that $u$ is linear on $E$, all extremal points of $E$ lie on $\partial \Omega$, and $1\leq \dim E < \frac{n}{2}$; see Caffarelli \cite{[C1],[C2],[C3],[C5]}.
	Moreover, if $u$ is the Brenier solution in optimal transportation, i.e., $	\det D^2 u(x)= \frac{f(x)}{g\left( \nabla u(x) \right)} $ for bounded positive functions $f$ and $g$, and both $\Omega$ and $\nabla u(\Omega)$ are convex sets, then $\Sigma_u=\emptyset$; see \cite{[BRE],[C6]}.
	
	 	 Note that the key to proving the classical strong maximum principle for uniformly elliptic equations is Hopf's lemma, which relies on constructing suitable auxiliary functions on the outer balls of the coincidence sets.  For the Monge-Amp\`ere equation, we will replace the outer balls with sections and use a deformation of the original function to construct new auxiliary functions. These sections and functions usually reflect how the corresponding linearized equations change near the strict convex points when the Monge-Amp\`ere measure satisfies the doubling condition.
	 	
	 	 On the other hand, for any fixed line segment $L$ in $\Sigma_u$, we will study the deformation of $u$ along the direction of $L$ and construct convex solutions $v$ with  $\det D^2v= \det D^2 u$ such that $u$ touches $v$ from above on $L$, which also shows the necessity of avoiding non-strictly convex points

	%	there is no    strong maximum principle on
	%	a domain containing $\Sigma_u$.
	\begin{Theorem}\label{thm:smp nece}
		Suppose that $u$ is a convex function, and $0 < \lambda \leq \det D^2 u \leq \Lambda$ on $\Omega$. For each segment $L$ in $\Sigma_u$ and every $x_0 \in \mathring{L}$, there is a convex function $v \not\equiv u$ defined around $x_0$, touches $u$ from below at $x_0$ and satisfies $\det D^2 v = \det D^2 u$.
	\end{Theorem}
	%	We will use Example \ref{exa:1 JT} to show that the strong maximum principle generally does not hold at non-strictly convex points.
	
	We now recall the results regarding the size of the non-strictly convex set $\Sigma_u$.  	If $n=2$,   any convex solutions to $  \det D^2u \geq   \lambda >0 $ are always strictly convex; see \cite{[C3]}. When $n \geq 3$, Mooney \cite{[Mo1]} proved that,  {\sl if $0<\lambda \leq \det D^2u \leq \Lambda $, then the  $(n-1)$-dimensional Hausdorff measure $\mathcal{H}^{n-1}(\Sigma_u )=0$, and the complement of the closed set $\Sigma_u$ is connected.} We then prove the sufficiency of avoiding $\Sigma_u$; please also refer to Remark \ref{rem:example explanation} for further understanding.
	\begin{Theorem}\label{thm:smp 0}
		Suppose that $u  \geq v $, both $u$ and $v$ are generalized solutions to \eqref{eq:ma measure f}. If $u \not \equiv v$,  then $   u$  can touch $v$  only  on $\Sigma_u \cap \Sigma_v$.
	\end{Theorem}
	
	We further consider the strong maximum principle for the following equation
	\begin{equation}\label{monge ampere measure}
		\det D^2 u(x)=f(x) \cdot\frac{ F\left(x, u\right) }{G\left(\nabla u, u^* \circ \nabla u\right)}\text{ in } \Omega,
	\end{equation}
	where $u^*$ is the Legendre transformation of $u$ (see Section 2 for its definition),   $\Omega$ is a connected open set,  and $F$ and $G$ satisfy   
	\begin{equation}\label{eq:polynomial measure 1}
		0<\lambda  \leq F 	 \leq \Lambda, \quad	0< \lambda \leq  G \leq \Lambda, \quad 	F \text{ and } G \text{ are locally Lipschitz.}
	\end{equation}
	We first deal with the non-degenerate case, assuming that
	\begin{equation}\label{eq:polynomial measure 0}
		0<\lambda  \leq f 	 \leq \Lambda.
	\end{equation}
	Note that since convex functions are $C^1$ almost everywhere, the term  $\frac{1}{G(\nabla u, u^* \circ \nabla u)}$ is well-defined as $L_{loc}^{\infty}$ functions.
	%	 If we assume that \eqref{eq:polynomial measure 1} and  $\M u$ is absolutely continuous with respect to the Lebesgue measure, then \eqref{monge ampere measure} is well-defined.

	\begin{Definition}\label{def:solution 01}
		Assume the assumption \eqref{eq:polynomial measure 1} holds. A convex function $u$ is a generalized (Aleksandrov) solution to \eqref{monge ampere measure}  if
		\[\M u =f(x) \cdot\frac{ F\left(x, u\right) }{G\left(\nabla u, u^* \circ \nabla u\right)}dx.\]
		%		Assume the assumption \eqref{eq:polynomial measure 1} holds. Suppose that $u$ is a convex function and the Monge-Amp\`ere measure $\M u$ is absolutely continuous with respect to the Lebesgue measure. If the density of $\M u $ equals $f(x) \cdot\frac{ F\left(x, u\right) }{G\left(\nabla u, u^* \circ \nabla u\right)}$, then we say $u$ is a generalized solution to \eqref{monge ampere measure}.
	\end{Definition}
	Whenever we write an equation or inequality involving $\det D^2 u$, we assume it holds in the sense of measure (Aleksandrov) by regarding $\det D^2u$ as $\M u$, the monge-Amp\`ere measure (see Section 2 for its definition), and the terms on both sides as measures. Here, for two given measures $\mu$ and $\nu$, the inequality $\mu \geq \nu$ means that $\mu-\nu$ is a non-negative measure. Similar to Theorem \ref{thm:smp 0}, we have
		\begin{Theorem}\label{thm:smp 1}
		Suppose that $u  \geq v $, both $u$ and $v$ are generalized solutions to \eqref{monge ampere measure}. Assume that  assumptions \eqref{eq:polynomial measure 1} and \eqref{eq:polynomial measure 0} hold. If $u \not \equiv v$,  then $   u$  can  touch $v$  only on $\Sigma_u \cap \Sigma_v$.
	\end{Theorem}
%	\begin{Theorem}\label{thm:smp 1}
%		Suppose that $u  \geq v $, both $u$ and $v$ are generalized solutions to \eqref{monge ampere measure}. Then we have the following:
%		%		Let $\Omega_0= \{x \in \Omega:\;  u=v\}$ denote the coincidence set.
%		%		 then
%		%		\[ \left(\partial \Omega_0 \cap \Omega \right)  \subset \left(\Sigma_u \cap \Sigma_v\right).\]
%		\begin{enumerate}
%			\item On each connected component $\Omega_i$ of $\Omega \setminus (\Sigma_u \cap \Sigma_v)$, either $u\equiv v$ or $u >v$;
%			\item If one of $u$ and $v$ is strictly convex, then either $u\equiv v$ or $u > v$ in $\Omega$.
%		\end{enumerate}
%	\end{Theorem}

	\begin{Theorem}\label{thm:smp 3}
		Suppose that $u  \geq v $, both $u$ and $v$ are  $C^1$ and strictly convex.  Assume that  assumptions \eqref{eq:polynomial measure 1} and \eqref{eq:polynomial measure 0} hold, and
		\[     \det D^2 u(x) \leq f(x) \cdot \frac{F\left(x, u\right) }{G\left(\nabla u, u^* \circ \nabla u\right)}\text{ in } \Omega \]
		and
		\[  \det D^2 v(x) \geq f(x) \cdot \frac{F\left(x, v\right)}{G\left(\nabla v, v^* \circ \nabla v\right)}\text{ in } \Omega \]
		in the sense of mesure (Aleksandrov).  Then, either $u > v$ or $u\equiv v$.
	\end{Theorem}

	%Here, we only focus on the points $x_0 \in \Sigma_u^c$.
	% to ensure that the Hessian $D^2u(x)$ is not degenerate.
	%	In general, the strong maximum principle fails on $\Sigma_u$. In Example \ref{exa:non-dege}, we construct convex solutions  $u$ and $ v$ to $\M u=\M v =1$ such that the coincidence set $\Omega_0= \{x \in \Omega:\;  u=v\} $ is inside  $\Sigma_u \cap \Sigma_v$, and the Hausdorff dimension of $\Omega_0$ can be chosen arbitrarily close to $n-1$.
	
	In Section \ref{chp:4},
	we present the corresponding version of Theorem \ref{thm:smp 3} in Minkowski problems. Using the Gauss image map $\boldsymbol{\alpha}_K(\cdot )$, the Minkowski problems for convex bodies $K \subset \R^{n+1}$ with $0 \in \mathring{K}$, are described by the equations of the type
	\[  \tilde{g}\left(\boldsymbol{\alpha}_K(\theta )\right) \frac{d \boldsymbol{\alpha}_K(\omega)}{d\omega} =\tilde{f}(\theta ) \frac{\tilde{F}(h)}{\tilde{G}(\rho)}\text{ for } \theta  \in \mathbb{S}^n,\]
		where  ${h}$ and $\rho$ denote the support function and radial function of $K$, respectively. The function
		$h$ can naturally be extended to a one-homogeneous function on $\R^{n+1}$, and by restricting it to a hyperplane, we obtain an equation of the form
	\[ g(\nabla u)\det D^2 u(x)=f(x) \cdot\frac{ F\left(x, u\right) }{G\left(\nabla u, u^* \circ \nabla u\right)}\]
	and a parallel version of Theorem \ref{thm:smp 3}, see Theorem \ref{thm:smp Gauss image} in Section \ref{chp:4}.	
	
	 Huang, Lutwak, Yang, and Zhang \cite{[HLY]}  studied the relationship between the $L_p$ Brunn-Minkowski theory and the $L_p$ dual Brunn-Minkowski theory for convex bodies $K \subset \R^{n+1}$ with $0 \in \mathring{K}$.  The  $L_p$ dual Minkowski 	 problem  was reduced to solving  the equation
	 \begin{equation}\label{eq:lp dual Minkowski problem}
	 	g\left(\frac{{\nabla}_{\mathbb{S}^{n}} {h}+{h}(\xi ) \xi  }{\sqrt{|\nabla_{\mathbb{S}^{n}} {h}|^2+{{h}}^2}}\right)
	 	\frac{\det\left({\nabla}_{\mathbb{S}^{n}}^2{h}+{h}\I \right)  }
	 	{\left({{h}}^2 +|{\nabla}_{\mathbb{S}^{n}}   {h} |^2\right)^{\frac{n+1-q}{2}} }  =f(\xi   ) {{h}}^{p-1} \text{ for } \xi \in \mathbb{S}^{n},
	 \end{equation}
 which is commonly expressed as  $	\frac{d \boldsymbol{\alpha}_K^*(\omega)}{d\omega }= \frac{1}{g \circ \boldsymbol{\alpha}_K^*}  \cdot  \frac{  {h^*}^{q}}{  {\rho^*}^{p}   }f(\xi)$ in terms of the dual body $K^*$ of $K$.  A fundamental topic is the existence, uniqueness, and regularity of convex body solutions  to  \eqref{eq:lp dual Minkowski problem}. However, there are almost no uniqueness results except for the case of the $g$ being a positive constant.  When $g$ is a positive constant, $q = n + 1$ and $ p \geq  1$, %the $L_p$ dual Minkowski problem is the $L_p$ Minkowski problem in centroaffine geometry, and
	the uniqueness of the solution was proved by \cite{[CW],[HuLY]} independently under the assumption that  $\mu:=f(x)dx \in \mathrm{NCH}$, i.e., $\mu$ is not concentrating on any hemisphere; but for $p\in (-n-1, 0)$ and any smooth positive functions $f$, at least two solutions were constructed in \cite{[HLW],[JLW]}.  When $g$ is a positive constant and the given measure $\mu$ is discrete, the uniqueness of the solution was proved by \cite{[LYZ]} for the case $p>q$,  and by \cite{[BF]} for the case $p>1, q>0$ and  $\mu \in \mathrm{NCH}$ being discrete. We refer to \cite{[CL],[HZ]}  for smooth solutions to the $L_p$ dual Minkowski problem.
	When $p \geq q$,   $g\equiv 1$, and $f \in C^{\alpha}$ is positive, Huang-Zhao \cite{[HZ]} proved the uniqueness of the smooth solution to the $L_p$ dual Minkowski problem  using the classical strong maximum principle. By Theorem 1.5, we will obtain the following result.
	\begin{Theorem}\label{thm:Lp DMP 0}
		Assume that $0<\lambda \leq f \leq \Lambda$, and $g $ is a positive Lipschitz function. Let $K_1$ and $K_2$ be convex bodies, both solving the $L_p$ dual Minkowski problem \eqref{eq:lp dual Minkowski problem} with $0 \in \mathring{K}_1\cap \mathring{K}_2$. If $p >q$, then $K_1$ equals to $K_2$; if $p=q$, then $K_1$ equals to $K_2$ up to dilation.
	\end{Theorem}

	Next, we investigate the degenerate case, for which we assume
	\begin{equation}\label{eq:polynomial measure 4}
		\lambda \sum_{i=1}^m   |\P_i(x)|^{\alpha_i} \leq f (x)	 \leq \Lambda  \sum_{i=1}^m      |\P_i(x)|^{\alpha_i},  \
		\P_i \text{ are polynomials},\  \alpha_i \geq 0.
	\end{equation}	
	Note that  $\sum_{i=1}^m   |\P_i(x)|^{\alpha_i}dx$   satisfies the doubling condition (see Definition \ref{def:doubling constant} in Section \ref{chp:2}).  It follows from
	 Caffarelli and Guti\'errez \cite{[CG]} that if $u$ is convex and $\det D^2 u \approx \sum_{i=1}^m   |\P_i(x)|^{\alpha_i} $,  then  $\Sigma_u  $ is closed, and $u\in C_{loc}^1(\Omega \setminus \Sigma_u)$.
In Section \ref{chp:5} we will prove
	%	   Examples of measures satisfying this doubling property include those whose density is the non-negative power of the norm of polynomial in $\R^n$. Moreover, it also extends to some negative powers of the norm of polynomials since we are only interested in the local case.
	
	\begin{Theorem}\label{thm:smp dege 1}
		Suppose that $u  \geq v $, both $u$ and $v$ are generalized solutions to \eqref{monge ampere measure}.   Assume the assumptions \eqref{eq:polynomial measure 1} and \eqref{eq:polynomial measure 4} hold. Then on each connected component $\Omega_i$ of $\Omega \setminus (\Sigma_u \cap \Sigma_v)$, we have either $u\equiv v$ or $u >v$.
	\end{Theorem}
	
	% Therefore, we need to understand the structure of the set of non-strictly convex points.

	  In Section \ref{chp:6}, we study the connectivity of the set $\Omega \setminus (\Sigma_u \cap \Sigma_v)$ appearing in Theorem \ref{thm:smp dege 1}. Our strategy is to use Mooney's delicate estimation of the sections at non-strictly convex points \cite{[Mo1]} to control the dimension of the set of non-strictly convex points.
	
	  \begin{Definition}\label{def:ZL sets}
	  Supposing that $\det D^2 u \geq \eta $ in $\Omega$, where $\eta \geq 0$ is a continuous function. We define the null set
	  \[\mathcal{Z}_{\eta}= \{ x \in \Omega :\; \eta(x)=0\}.\]
	  Let
	  \[\L_{\eta}=\bigcup_{L \subset \mathcal{Z}_{\eta}}L, \text{ where } L \text{ is a line segment but is non-singleton}, \]
	  and
	  \[ \begin{split}
	  	\tilde{\L}_{\eta} = \bigcup_{E \subset \mathcal{Z}_{\eta}}E,
	  	& \text{ where } E \text{ is a convex set but not a singleton}, \\
	  	&\text{ and all extremal points of $E$ are on the boundary $\partial \Omega$. }
	  \end{split}\]
	  Clearly, $\tilde{\L}_{\eta}  \subset \L_{\eta}$.
	  \end{Definition}
	%	\[ \L_{\eta}=\left\{x \in \mathcal{Z}_{\eta}:\; \{x\} \subsetneq L \text{ for some line segments } L \subset  \mathcal{Z}_{\eta} \right\}\]
	%	and
	%	\[ \begin{split}
		%		\tilde{\L}_{\eta} =\{x \in \mathcal{Z}_{\eta}
		%		&:\;x\in  E \text{ for some convex set } E \subset  \mathcal{Z}_{\eta} \\
		%		&\text{ such that all extremal points of $E$ are on  } \partial \Omega \}.
		%	\end{split}\]
	\begin{Theorem}\label{thm:hausdorff n-1}
		Let $u $ be a convex solution to $\det D^2 u \geq \eta $ in $\Omega$,  where $\eta \geq 0$ is a continuous function.    Then,
		\[  \mathcal{H}^{n-1}(\Sigma_u \setminus \L_{\eta}  )=0 .\]
		Furthermore, if $\M u $   satisfies the doubling condition, then
		\[\mathcal{H}^{n-1}(\Sigma_u \setminus \tilde{\L}_{\eta} )=0.\]
	\end{Theorem}
	
	Combining Theorems \ref{thm:smp dege 1} and \ref{thm:hausdorff n-1}, we have
	
	\begin{Theorem}\label{thm:smp dege 2}
		Suppose that $u  \geq v $, both $u$ and $v$ are generalized solutions to \eqref{monge ampere measure}.   In addition to the assumptions \eqref{eq:polynomial measure 1} and \eqref{eq:polynomial measure 4}, we further assume that $\mathcal{H}^{n-1}(\tilde{\L}_{\eta}  )=0$, where $\eta (x)=\sum_{i=1}^m   |\P_i(x)|^{\alpha_i}$. If  $u \not\equiv v$,  then $   u$  can   touch $v$ only on $\Sigma_u \cap \Sigma_v$.
	\end{Theorem}
	
	If the zero set of $\eta$ consists of strictly convex hypersurfaces, we have ${\L}_{\eta}=\emptyset$. And in Example \ref{exa:dege} , we will construct convex solutions $u$ and $v$ to
	\[  \  \det D^2 u =\det D^2 v =  \left|x_n -|x''|^2\right|^{ \alpha}  \text{ for } x=\left(x'',x_{n-1},x_n\right) \in  \R^{n-2} \times \R\times \R ,  \]
	such that $u\geq v $, $\Sigma_u=\Sigma_v= \left\{x_n =|x''|^2 \right\}$.  %with the coincidence equals $  \left\{x_n \leq |x''|^2 \right\}$.

	Finally, using the same method as Mooney \cite[Theorem 1.2]{[Mo1]}, Theorem \ref{thm:hausdorff n-1} implies the $W^{2,1}$ regularity of convex solutions to  Monge Amp\`ere equations, as derived from  Philippis, Figalli and Savin \cite{[PhiF],[Phi1]}. 	 
	
	\begin{Theorem}\label{thm:w21 estimate}
		Assuming $u$ is a convex solution of
		$\det D^2 u =f  \text{ in } \Omega $
		with $f$ as in \eqref{eq:polynomial measure 4}.  If $\mathcal{H}^{n-1}(\tilde{\L}_{\eta}  )=0$, where $\eta (x)=\sum_{i=1}^m   |\P_i(x)|^{\alpha_i}$, then $u \in W_{loc}^{2,1}(\Omega)$.
	\end{Theorem}
	
It is worth noting that Mooney also showed that the optimal regularity for solutions of $\det D^2 u \geq 1$ is that the second derivatives are in $L\log^{\varepsilon}L$ for small $\varepsilon>0$. On the other hand, there exist $C^{1,\alpha}$ strictly convex solutions of $\det D^2 u \leq 1$ that are not even in $W^{2,1}$, as shown in \cite{[PhiT]}.

	This paper is organized as follows: In  Section \ref{chp:2}, we introduce some   notations,   and preliminaries concerning the Monge-Amp\`ere measures.
	In  Section \ref{chp:3}, we prove Theorems \ref{thm:smp 1} and \ref{thm:smp 3}.
	In   Section \ref{chp:4},
we  first study the convex body version corresponding to Theorem \ref{thm:smp 3} via the Gauss image map, namely Theorem \ref{thm:smp Gauss image}, and then prove Theorem \ref{thm:Lp DMP 0}.
	In Section \ref{chp:5}, we first explore a nice property about polynomials (which will be called {\sl EDT property}), then use this property to construct auxiliary functions and prove Theorem \ref{thm:smp dege 1}.
	In Section \ref{chp:6}, motivating by Mooney's estimation of the size of sections at non-strictly convex points, we prove Theorems \ref{thm:hausdorff n-1} and \ref{thm:w21 estimate}. Finally in Appendix \ref{chp:a},  
 we will construct a few examples to show that   how large the coincidence can be, and how close it is to the set $\Sigma_{u}$ and $\Sigma_{v}$.

	\section{Preliminaries}\label{chp:2}

	In this paper,   $U$ and $V$ are bounded, open sets;  $u, v$ and $w$ are convex functions.  We will use $c$ and $ C$ to represent positive universal constants that depend only on $n, \lambda,\Lambda$, the Lipschitz norms of $F$ and $G$, and $b$, the doubling constants of $fdx$ in the Definition \ref{def:doubling constant} below. For any sets $A  ,B\subset \R^n$ and constant $a $, we define $aA=\{ ax:\; x\in A\}$, $A+B=\{ x+y:\; x\in A \text{ and } y \in B\}$. We will regard $x$ as $\{x\}$ when there is no confusion. All the segments and convex sets are usually assumed to be non-singleton, i.e., not a single point set.
	
	We refer to \cite{[CG],[F],[P],[TW1]} for following discussions on Monge-Amp\`ere measures.
	Given a convex function $u$ on $\Omega$, its subgradient at $x_0 \in \Omega$ is
	\[
	\partial u(x_0)=\left\{p \in \mathbb{R}^{n} :\; u(x) \geq u(x_0)+p \cdot(x-x_0) \text{ for all } x \in \Omega\right\},
	\]  and $u(x_0)+p \cdot(x-x_0) $ is a supporting hyperplane function of $u$ at $x_0$ if $p \in \partial u(x_0)$.   The Legendre transformation of $u$ is the convex function
	\[u^*(p) = \sup_{x \in \R^n}\{ x \cdot p -u(x)\} \text{ for } p \in \R^n,\]
	where  we have extended $u$  to be positive infinite outside $\Omega$ for convenience.  $\partial u(E):=\bigcup_{x \in E} \partial u(x)$ is measurable for any open and closed subset $E \subset \subset \Omega$, and so for any Borel set $E \subset \subset \Omega$. The  Monge-Amp\`ere measure $\M u$ is defined by $\M u(E)=|\partial u(E)|$ for each Borel set $E \subset \Omega$, $u$ is called the generalized (Aleksandrov) solution to $\det D^2 u = f $ if $\M u= fdx$.

	As introduced in \cite{[C4]}, let $\kappa>0$, a bounded convex set $E\subset \R^{n}$ is $\kappa$-balanced about point  $x$ if
	\[t(x-E)\subset E-x  \text{ for all } t \in [0,  c\kappa].  \]
	And $E$ is balanced about $x$ if $\kappa $ is universal.
	
	\begin{Lemma} [\cite{[John1]}, John's Lemma]%2.3
		Suppose  $\Omega \subset \R^n$ is a bounded convex set and $\mathring{\Omega} \neq \emptyset$. Then there is an ellipsoid  $E$ (called the John ellipsoid, the ellipsoid of the maximal volume contained in $\Omega$) such that
		\[ E  \subset  \Omega  \subset C(n) \left(E-x_E\right)+x_E,\]
		where $ x_E$ is the mass center of $E$ and $C(n)$ is a positive constant only depending on $n$.
	\end{Lemma}

	%\[
	%\partial u(x)=\left\{p \in \mathbb{R}^{n}: u(y) \geq u(x)+p \cdot(y-x) \quad \text { for all } y \in \Omega\right\}
	%\]

	%Similarly, in the paper, we recall the following definition.

	We recall the comparison principle for generalized solutions. See \cite{[Ba1]}, \cite{[QT]}.
	
	\begin{Lemma}[Comparison Principle]\label{lem:comparison principle 01}
		Suppose that $u,v \in C(\overline{U})$, satisfying
		\[
		u \geq v \text{ on } \partial U  \text{ and }   \M u(E) \leq  \M v(E) < \infty \text{ for all Borel set }  E \subset U.
		\]
		Then $
		u(x) \geq v(x)   \text{ in } U.
		$
	\end{Lemma}
	
	%	\begin{Lemma}[Comparison Principle]\label{lem:comparison principle 01}
		%		Let  $\nu$ be an absolutely continuous Borel measure on $U$ with positive density.
		%		
		%		Suppose that $u,v \in C(\overline{U})$ are convex functions and $u \geq v$. If
		%		\[
		%		|\partial u(E)| \leq |\partial v(E)| < +\infty\text { for every Borel set } E \subset U .
		%		\]
		%		Then,
		%		\[
		%		u(x) \geq v(x), \text{ for all } x \in U \text {. }
		%		\]
		%	\end{Lemma}
	
	\begin{Definition}[\cite{[P]}, Sections]\label{def:sections}
		Let $u$ be a convex function on $U$. The section of $u$ at based point $x_0 \in U$ with height $t$  for subgradient $p \in \partial u(x_0)$ is
		\[ S_{t,p}^u(x_0) = \left\{ x \in {U}  :\; u(x) < u(x_0)+p\cdot (x-x_0)+t  \right\}.\]
		$S_{t,p}^u(x_0)$ is an  internal section provided that $S_{t,p}^u(x_0) \subset \subset U$. If $u$ is $C^1$ at $x_0$, then $\partial u(x_0) = \{ \nabla u(x_0)\}$, and we will write the section as $S_{t}^u(x_0)$ .
	\end{Definition}

	The following is the well-known doubling condition for  Monge-Amp\`ere measures.
	\begin{Definition}[\cite{[CG]}, Doubling measure for a given function]\label{def:doubling measure}
		Given a Borel measure $\mu$ on $U$ and a convex function on $U$.  Suppose that there are constants $k >0, 0<\alpha<1$ satisfying
		\[
		\mu \left(S_{t,p}^u(x)\right) \leq k \mu\left( \alpha  \left(S_{t,p}^u(x)-x\right)+x\right) \text{ for any } S_{t,p}^u(x) \subset \subset U.
		\]
		Then, we say that $\mu$ satisfies the doubling condition with respect to $u$ on  $U$,  and  $\kappa,\alpha$ are doubling constants of $\mu$ for $u$.
	\end{Definition}

	\begin{Lemma}[\cite{[CG]}, Engulfing property]\label{def:engulfing condition}
		Suppose that $\M  u $ satisfies the doubling condition with respect to $u$ on $U$.  Then, every internal section $S_{t,p}^u(x)$ is $\kappa$-balanced about its based point $x$, where $\kappa>0$ depends on the doubling constants of $\M  u $ for $u$. Moreover, if $y \in S_{\kappa t,p}^u(x)$ and   $S_{t,p}^u(x) \subset \subset U$, then $S_{\kappa^2t,p}^u(x)\subset S_{\kappa t,q}^u(y)$ for every $q \in \partial u(y)$.
	\end{Lemma}
	
	The following  Lemmas \ref{lem:Pogorelov Line} and \ref{lem:Pogorelov Line 1} below can be found in \cite{[C1],[C2],[C3]}.
	
	\begin{Lemma} \label{lem:Pogorelov Line}
		Suppose that $\M  u $ satisfies the doubling condition with respect to $u$ on $U$. Then, for every internal section $S_{t,p}^u(x)  $, we have
		$u \in C_{loc}^{1,\alpha} \left(S_{t,p}^u(x)\right),$
		where $\alpha>0$ depends only on $n$ and the doubling constants of $\M u$ for $u$. Thus,   $\Sigma_u$ is closed,  and $u$ is  $C^1$ and strictly convex on $U\setminus \Sigma_u$. Moreover, $\Sigma_u$  is the union of
 the all  such convex sets $E $  that $u$ is linear on $E$, all extremal points of $E$ are on $\partial U$, and $1\leq \dim E < \frac{n}{2}$. Here and below,  a point $x $ is called the extremal point of the convex set $E$ if
  $x \in \partial E$ and $H \cap E =\{x\}$ for some hyperplane $H$.  When $n = 3$ or 4, such  $E$ are line segments.
	\end{Lemma}

	\begin{Lemma}\label{lem:Pogorelov Line 1}
		Suppose that $u$ is convex in $U$ and $x_0\in U$, then
		\[x_0 \notin  \Sigma_u
		\Leftrightarrow \bigcap_{t> 0}S_{t, p}^u(x_0)= \{x_0\}, \; \forall \; p\in \partial u(x_0).\]
		Furthermore, if $\M u $  satisfies the doubling condition on $U$, the strict convexity of $u$ at $x_0$ is equivalent to the existence of an internal section at $x_0$.
	\end{Lemma}

	%If the measure is given by the non-negative power of norms of polynomials in $\R^n$, then it satisfies the doubling condition
	%		It also extends to some negative powers of the norm of polynomials since we are only interested in the local case.
	
	\begin{Definition}[\cite{[CG]}, Doubling condition]\label{def:doubling constant}
		Let  $\mu$ be a Borel measure on $U$ and $c_n$  be a small universal constant that depends only on $n$. Suppose that there is a small constant $b>0$  such that for every convex set $E\subset U$, we have
		\[
		\mu \left(c_n(E-x_E)\right) \geq b\mu (E-x_E),
		\]
		where $x_E$ denotes the mass center of $E$. Then we say $\mu$ satisfies the doubling condition and call $b$ the doubling constant of $\mu$.
	\end{Definition}

	To check the doubling condition for a given convex function, following \cite[Lemma 2.1]{[CG]},   we only need to check it for all ellipsoids. See \cite[Lemma 2.3]{[JW2]} for its proof.
%Recall thatfor any convex solution $u$ to $\M u =\mu$ on $U$, $\mu$ satisfies the doubling condition.

	\begin{Remark}
		Note that both $ F$ and $G $  satisfy \eqref{eq:polynomial measure 1} in this paper. For a solution $u$ to \eqref{monge ampere measure}, we always regard the doubling condition for $fdx$ as the doubling condition for $\M u$.
	\end{Remark}

	Polynomials satisfy the doubling condition. In fact, they satisfy a stronger condition, the so-called {\sl condition $\left(\mu_{\infty}\right)$}.
	
	\begin{Lemma}[\cite{[Phi1]}, Condition $\left(\mu_{\infty}\right)$]\label{def:condition mu}
		Suppose that    $f$ satisfies $\lambda|\P(x)|^{\alpha}  \leq f(x) \leq \Lambda|\P(x)|^{\alpha}$ 
for some constant $\alpha \geq0 $  and a polynomial $\P(x)$. Let $\mu =fdx$. Then there are constants $\beta \geq 1$ and $\gamma>0$, depending only on $\lambda,\Lambda,n,\P,\alpha$, such that for all convex sets $E$,
		\[
		\frac{\mu(V)}{\mu(E)} \geq \gamma\left(\frac{|V|}{|E|}\right)^\beta   \text{ holds for all open set } V \subset E.
		\]
	\end{Lemma}
	
	The following Lemmas are well known for convex functions.
	\begin{Lemma}[\cite{[F],[G], [TW1]}, Aleksandrov-Bakelman-Pucci Maximum Principle]\label{lem:size section abp}
		%		Suppose $u$ is a convex function on $U$.
		If $S_{t,p}^u(x_0) $ is an internal section, then
		\[ | \M u (S_{t,p}^u(x_0))| \cdot |S_{t,p}^u(x_0)| \geq c_n t^{n}, \]
		where positive constant $c_n $ only depends only on $n$.
	\end{Lemma}
	
	The inverse inequality holds  if $\M u$ satisfies  the doubling condition.
	\begin{Lemma}[\cite{[CG]}, Lemma 1.1]\label{lem:size section 2}
		Suppose that $u \in C(\overline{U})$ and
		%		$u$ is a convex function and
		$\M u$ satisfies the doubling condition, then for any convex set $E \subset U$,
		\[  |\M u (E) |\cdot |E| \leq C \inf_{\ell \text{ is linear }} ||u-\ell ||_{L^{\infty} (E)}^n, \]
		where the infimum is taken for all  linear function $\ell$  and $C$ depends only on $n$ and the doubling constant.
	\end{Lemma}

	The following Lemma \ref{lem:size section 3} is a  modification of \cite[Lemma 2.2]{[Mo1]}.
	\begin{Lemma} \label{lem:size section 3}
		Suppose that $u \in C(\overline{U})$, $\M u \geq \mu$, and $\mu$ satisfies the doubling condition. Then  the
conclusion of \ref{lem:size section 2} still holds true,
		where the $C$ depends only on $n$ and the doubling constant of $\mu$.
	\end{Lemma}
	\begin{proof}
		Let $E$ be a convex set and $\ell$ be a linear function. Let $w$ be a convex solution to the Dirichlet problem
		\[ \det D^2 w =\mu \text{ in } E \text{ and } w = u \text{ on } \partial E.\]
		By comparison principle, $w-\ell \geq u-\ell $ in E.  Note that $w$ is convex, and
		\[  \sup\{w(x)-\ell(x) :\; x\in E\}\leq \sup\{w(x)-\ell(x) :\; x\in \partial E\}=\sup\{u(x)-\ell(x) :\; x\in \partial E\}. \]
		Therefore,
		\[||w-\ell||_{L^{\infty} (E)} \leq ||u-\ell||_{L^{\infty} (E)} .\]
		The proof now follows from Lemma \ref{lem:size section 2}.
	\end{proof}

	%
	%
	%Suppose that $\Omega \in R^n$ is a convex domain. Let $u,v \in C(\overline{\Omega})$ be two convex solutions to
	%\begin{equation}\label{monge ampere measure equiv}
	%	\det D^2 u=\det D^2 v =f   \text{ in } \Omega,
	%\end{equation}
	%where $ \lambda \leq f  \leq \Lambda$.
	%%We also assume
	%%\begin{equation}\label{meausre small perturb}
	%%	\omega(f,0^+) =  \lim_{r \to 0} \sup \{\log f(x)-\log f(y) :\;x,y \in B_r(0)\}  \leq  \epsilon(n)
	%%\end{equation}
	%%for some  sufficient small constant $\epsilon(n) >0$  if $n \geq 3$.
	%

	\section{The Strong Maximum Principle for Monge-Amp\`ere equations}\label{chp:3}
	
		 We recall the existence of solutions to the Dirichlet problem of Monge-AmpA\`ere equations, the original version is given by Aleksandrov, see \cite{[F],[G]}.
	\begin{Lemma}[Aleksandrov]
		Suppose $\Omega \subset \R^n$ is an open bounded convex set, $\mu$ is a finite Borel measure on $\Omega$, and $g \in C(\overline{\Omega})$ is convex. Then there exists a unique convex solution $u \in C(\overline{\Omega})$ of
		\[	\operatorname{det} D^2 u=\mu \text{ in } \Omega, \quad u =g \text{ on } \partial \Omega.\]
	\end{Lemma}
	
	\begin{proof}[Proof of Theorem \ref{thm:smp nece}]
	We write points $x\in \R^{n}= \R^{n-1} \times \R$ as $x=(x',x_n) $.	Without loss of generality, we assume that  $x_0=0 \in \Sigma_u$, $u \geq 0$,  $ u(t e_n) =0 $ when $|t|$ is small.
	Let $M>0$ be a large constant, and consider the function
		\[	   w :=\frac{  u(x)+ \delta u\left(\frac{1}{2} x+   M x_n  e_n  \right) }{1+\delta}. \]
		Then we have
		\[ \det D^2 w \geq  \det D^2 u  \text{ and } w \leq  \frac{2+\delta}{2+2\delta}u \leq u \text{ in } B_{c/M}(0).\]
		By lemma 3.1, let $v$ be the solution to
		\[ \det D^2 v = \det D^2 u  \text{ in } B_{c/M}(0), \quad v=w \text{ on }  B_{c/M}(0).\]
		Note that $\det D^2 w \geq\det D^2 v =\det D^2 u$, the comparison principle implies that
		\[   w \leq v \leq u  \text{ in } B_{c/M}(0). \]
		Hence, $u$ touches $v$ from above on $x'=0$ around $0$. Clearly, we have $v \not \equiv u$.
	\end{proof}
	
	\begin{Remark}\label{rem:example explanation}
		 In the proof of Theorem \ref{thm:smp nece}, let $E:=\left\{x:\;u(x)=v(x)\right\}$ denote the coincidence set, then we have $\partial E \subset \left\{u(x)=0\right\} \cap B_{c/M}(0)$. Assuming that Theorem \ref{thm:smp 0} is true, we have  $E \subset \Sigma_{u} $, which implies that $E \subset \left\{u(x)=0\right\} $.
	\end{Remark}
	
	\begin{Lemma}\label{lem:measure constrain double equiv}
		Suppose that $u$ and $v$ are convex functions on $\Omega$, $u \geq v$, $ \M u \leq  C \M v $, and both $\M u$ and $\M v$ satisfy the doubling condition. Then
		\[  \left(\Sigma_u \cap \Omega_0\right) = \left(\Sigma_v \cap \Omega_0\right),
		\]
		where $\Omega_0= \{ x \in \Omega :\; u(x)=v(x)\} $ is the coincidence set.
	\end{Lemma}
	\begin{proof}
		Suppose $x_0 \in  \Omega_0$. Since $u \geq v$ and $u(x_0)=v(x_0)$,  we have $\partial v(x_0) \subset \partial u(x_0)$,
		and  for any subgradient $p \in  \partial v(x_0)$ and any $ t>0$, we have
		\[ S_{t,p}^u(x_0) \subset S_{t,p}^v(x_0) . \]
		
		If $x_0 \notin \Sigma_v $, then $v$ is $C^1$ and strictly convex at $x_0$. If $t>0$ is small, then
		\[S_{t,\nabla v(x_0)}^u(x_0) \subset S_{t,\nabla v(x_0)}^v(x_0) \subset\subset \Omega  . \]
		Recalling Lemma \ref{lem:Pogorelov Line 1},  we find that $x_0 \notin \Sigma_u$.

		If $x_0 \notin \Sigma_u$, then $u$ is  $C^1$ and strictly convex at $x_0$. By noting that $\emptyset \neq  \partial v(x_0) \subset \partial u(x_0) = \{ \nabla u(x_0)\}$, we have $\partial v(x_0) =\partial u(x_0) =\{ \nabla v(x_0)\}=\{ \nabla u(x_0)\}$. By choosing $t>0$ sufficiently small such that $S_{t}^u(x_0)$ is an internal section, we can apply Lemmas \ref{lem:size section abp} and \ref{lem:size section 2} to obtain that
		\[ c| \M v (S_{t}^v(x_0))| \cdot |S_{t}^v(x_0)|
		\leq t^{n}
		\leq  C| \M u (S_{t}^u(x_0))| \cdot |S_{t}^u(x_0)|
		\]
		Note that
		\[| \M u \left(S_{t}^u(x_0)\right)| \leq | \M u \left(S_{t}^v(x_0)\right)| \leq C| \M v \left(S_{t}^v(x_0)\right)|.\]
		Thus,
		\[ |S_{t}^v(x_0) | \leq C |S_{t}^u(x_0) |.  \]
		Since internal sections are convex sets, we conclude that
		\[  S_{t}^u(x_0) \subset S_{t}^v(x_0) \subset C S_{t}^u(x_0),\]
		where $C$ is independent of $t$. Recalling Lemma \ref{lem:Pogorelov Line 1} again,  we have
		\[ x_0 \notin \Sigma_u
		\Rightarrow \bigcap_{t> 0}S_{t}^u(x_0)= \{x_0\} \Rightarrow \bigcap_{t> 0} S_{t}^v(x_0)= \{x_0\}
		\Rightarrow x_0 \notin  \Sigma_v.\]
	\end{proof}

	\begin{proof}[Proof of Theorem \ref{thm:smp 0}]
		Let $\Omega_0= \{x \in \Omega:\;  u=v\}$ denote the coincidence set, it suffices to show that $\left(\partial \Omega_0 \cap \Omega \right)  \subset \Sigma_{u}\cap \Sigma_v  $. 
Otherwise, without loss of generality, we assume   that $ \left(\partial \Omega_0 \cap \Omega \right)  \not \subset \Sigma_v$. Then we can take a point $X \in \partial \Omega_0 \cap \Omega\setminus \Sigma_v $.
According to Lemma \ref{lem:measure constrain double equiv}, we can always assume that both $u$ and $v$ are $C^{1}$ and strictly convex near $X$ in the following discussion.
		By the engulfing property of sections, we have
		\[S_{ct_0}^{v}(X) \subset S_{t_0}^{v}(Y) \subset S_{Ct_0}^{v}(X) \subset \subset\Omega \]
		provided that $Y \in  S_{c^2t_0}^{v}(X)$ and  $t_0< t(X,u)$ is small enough. We now choose $Y \in  S_{c^2t_0}^{v}(X) \setminus  \Omega_0$ and $ \bar t=  \sup\{ t :\; S_t^{v}(Y) \subset \subset \Omega \setminus \Omega_0\}$.  Since $X\in S_{t_0}^{v}(Y)\cap \partial \Omega_0$, we have $\bar{t} < t_0 $ and $S_{\bar t}^{v}(Y)$ is an internal section.
		%	
		%	We can assume that $\bar t $ is small and  $u, v$  are  strictly convex and   $C^{1,\alpha}$ on $ S_{c\bar t}^{v}(Y) \subset S_{\bar t}^{v}(X)$.
		%	
		
		Hence, without loss of generality,  we now assume that $\partial S_{\bar t}^{v}(Y)$ touches $  \Omega_0$ at point $0$,
		\[ u(0)=v(0)=0,\; \nabla u(0)=\nabla v(0)=0,\;  u,v \geq 0,  \text{ and } \nabla v(Y) = c_1e_n \text{ for } c_1 > 0.\]
		Thus, $S_{\bar t}^{v}(Y) =\{ x:\; v(x ) \leq c_1 x_n\}$.   Let
		\[  U=\{ x:\; v(x) < c_2 x_n\},\]
		where $c_2 \in (0,c_1)$ is a small constant to be fixed later. We have
		\[ U\subset S_{\bar t}^{v}(Y)  \subset \{ x_n > 0\} \text{ and } \partial U \cap  \partial S_{\bar t}^{v}(Y)  \subset    \{ v(x) = 0\}= \{0\},\]
		and $\partial U$ is represented by a  $C^1$ and strictly convex function around   $0$. Define
		\[	\sigma(s ) :=\inf\left\{ u(x) -v(x)  :\; x\in \partial U,   x_n \geq  s   \right\} \text{ for } s>0,\]
		$\sigma(s )$ is positive and satisfies $\lim_{s\to 0} \sigma(s)=0$.
		
		Let $M$ be a large constant, we consider the function
		\[\Phi (x) = v\left( \T x  \right),  \text{ where }\T x = \frac{1}{2} x+   M x_n  e_n   .\]
		Clearly,
		\begin{equation}\label{eq:det u geq w U}
			\det  D^2 \Phi(x) \geq
			c(n) {M}^2 \det D^2 v(\T x) =c(n) {M}^2   f(\T x) \geq c(n,\lambda,\Lambda) {M}^2 f(x).
		\end{equation}
		By the $C^1$ regularity of $v$ at $0$, we have
		\[  \begin{split}
			\Phi (x)
			& =  v\left(\frac{1}{2} x+   M x_n e_n  \right)\\
			&\leq  \frac{1}{2} v(x)+\frac{1}{2} v\left(2M x_n e_n \right)\\
			& =  \frac{1}{2} v(x) + o(|x_n|)  \text{ on } \partial U\cap \{    x_n \leq  s_0 \}
		\end{split}
		\]
		and
		\[u(x)  \geq  \frac{1}{2}u(x)+\frac{c_2}{2} x_n   \geq   \Phi(x)+\frac{c_2}{4} x_n  \text{ on } \partial U\cap \{    x_n \leq  s_0 \}\]
		provided that $s_0=s_0(c_2) >0 $ is small enough.

		%		\[ \sum_{j=1}^{l} \left|{\Q}_i'\left(Z_j'+\epsilon (y'-Z_j)\right)\right| \geq \epsilon^{\deg {\Q}_i}\left|{\Q}_i'\left(y'\right)\right|.  \]
		%		
		
		We now take  $\delta \in (0,  \sigma(s_0))$ and consider the convex function
		\[ w = \frac{ v+ \delta \Phi+ \frac{c_2\delta}{4}x_n}{1+\delta}. \]
	 By the definition of $\sigma$ we have
		\[ 		u \geq   w   \text{ on } \partial  U. \]
	Note that  $u,v$ and $w$ are all $C^1$ at $0$. We claim that if $c_2>0$ is small, then
		\begin{equation}\label{eq:u geq w U}
			u \geq   w   \text{ in }  \overline    U.
		\end{equation}
		This will imply that $\left(\nabla  u(0)- \nabla w(0)\right) \cdot e_n \geq 0$. Recalling that $\nabla u(0)= 0$ and $\nabla w(0)= \frac{c_2\delta}{4} e_n$, we obtain $c_2<0$,   contradicting  the fact $c_2>0$, and thus 
Theorems 1.2 is proven.

%		We first prove \eqref{eq:u geq w U} for the special case $F\equiv G\equiv 1$, which is to say that
Recalling the equation
		\[\det D^2 u \equiv \det D^2 v=f.\]	
		By taking a larger $M$, the arithmetic-geometric inequality now gives
		\[\begin{split}
			\det D^2w   \geq \left[\frac{\left( \det D^2 v\right)^{1/n} +\delta \left( \det D^2  \Phi\right)^{1/n}}{1+\delta} \right]^n
			\geq (1+c{M}^{2/n} \delta) f >\det D^2 u
		\end{split}\]
		in the sense of measure (Aleksandrov).  Hene \eqref{eq:u geq w U} follows from the comparison principle.

	\end{proof}
	\begin{Lemma}\label{lem:legendre t 1}
	Suppose that $u$ and $ w $ are strictly convex functions on $\Omega$, and  $V \subset \subset \Omega$. Assume that
	\[  u \leq w \leq u+\delta \text{ in } V \text{ and } u=w \text{ on } \partial V \]
	for some $\delta >0$. Then
	\[ \partial w (V) \subset  \partial u(V).\]
	Moreover,
	\[   u^* -\delta \leq w^* \leq u^* \text{ on } \partial w(V),  \]
	where $u^*$ and $w^*$ are the Legendre transformations of $u$ and $w$, respectively.
\end{Lemma}
\begin{proof} 
	Assuming that $x_0 \in V$ and $p \in \partial w(x_0)$, we can move the supporting hyperplane function
	$\ell(x)$ of $w$ at $x_0$ downward until it touches $u$ at some point $X \in V$. Then, $p \in \partial u(X)  \subset \partial u(V)$. Moreover,
	\[ u(X)\leq \ell(X):= (X- x_0 ) \cdot p- w(x_0) \leq w(X) \leq u(X)+\delta. \]
	Recalling the well-known fact
	\[ u^*(p)= X\cdot p-u(X) \text{ and } w^*(p)= x_0 \cdot p- w(x) ,\]
	we obtain
	\[  u^*(p) -\delta \leq w^*(p) \leq u^*(p) .  \]
\end{proof}

\begin{proof}[Proof of Theorem \ref{thm:smp 1}]
		 We follow the same proof technique as Theorem \ref{thm:smp 0}, and only need to verify \eqref{eq:det u geq w U}, that is to say the set 
		\[  V :=\{ x \in U:\;  u(x)<w(x)\} \]
is empty. We have in   $V$,
		 \begin{equation}\label{eq:smp p-1}
			v(x) \leq u(x) \leq w(x)  \leq v(x)+C\delta \leq u(x)+C\delta   .
		\end{equation}
		Applying Lemma \ref{lem:legendre t 1}, we now have
		\[\partial w( V ) \subset \partial u( V ),\]
		and
		\[  |w^* \circ \nabla w -u^* \circ \nabla u  | \leq C \delta \text{ in }   V.\]
		If  $c_2$ is small enough, then
		%		the $C^1$ regularity of $u$ and $v$ at $0$ implies that
		\[	|\nabla w(x) -\nabla v(x)|\leq \delta {M}^{2}\cdot o(|x|) +c_2\delta \leq  C\delta \text{ in }    V \]
		and
		\[	\begin{split}
			\left|w^*\circ\nabla w(x) -v^*\circ\nabla v(x)\right|
			& = \left| x\cdot \nabla w(x)-w(x)- \left( x\cdot \nabla v(x)-v(x)\right) \right|\leq C \delta.
		\end{split}\]
		
		Therefore, we have
		\[
		\det D^2 v(x)
		=  f(x) \frac{F\left(x,v\right)  }{G\left( \nabla v, v^*\circ \nabla v \right)}  \geq (1-C\delta)f(x) \frac{F\left(x,u\right)  }{G\left( \nabla w, w^*\circ \nabla w \right)} \]
		and
		\[ \begin{split}
			(1+\delta)\left( \det D^2 w(x)\right)^{1/n}
			& \geq \left( \det D^2 v(x)\right)^{1/n} +\delta \left( \det D^2  \Phi(x)\right)^{1/n}\\
			%			&
			%			\geq \left( \det D^2 v(x)\right)^{1/n} +c{M}^{2/n} \delta  \left( \det D^2 v(\T x)\right)^{1/n}\\
			& \geq \left((1-C\delta) f(x) \frac{F\left(x,u\right)  }{G\left( \nabla w, w^*\circ \nabla w \right)}\right)^{1/n} +cM^{2/n}\delta f^{1/n} (x)   \\
			%			& \geq \left( f(x) \frac{F\left(x,v\right)  }{G\left( \nabla v, v^*\circ \nabla v \right)}\right)^{1/n} +cM^{2/n}\delta  \left( f(x) \frac{F\left(x,v\right)  }{G\left( \nabla v, v^*\circ \nabla v \right)}\right)^{1/n}   \\
			%			& \geq  \left(1+c{M}^{2/n} \delta\right)\left( f(x) \frac{F\left(x,v\right)  }{G\left( \nabla v, v^*\circ \nabla v \right)}\right)^{1/n} \\
			& \geq  \left(1+c{M}^{2/n} \delta\right)\left(f(x) \frac{F\left(x,u\right)  }{G\left( \nabla w, w^*\circ \nabla w \right)} \right)^{1/n}
		\end{split}\]
		in the sense of measure (Aleksandrov). This is
		\begin{equation}\label{eq:theorem rem 2}
			G\left( \nabla w, w^*\circ \nabla w \right)\det D^2 w  \geq (1+c{M}^{2/n} \delta)f(x)F\left(x,u\right) .
		\end{equation}
		
		Hence, using the related results from \eqref{eq:smp p-1}-\eqref{eq:theorem rem 2} we have
		\begingroup
		\allowdisplaybreaks  \begin{align*}
			\int_{\nabla u( V )}G\left(p, u^*\right)dp
			& =	\int_{ V }G\left(p\circ \nabla u, u^*\circ \nabla u\right) \det D^2 u \cdot dx  \\
			& =	\int_{ V }f(x)F\left(x,u\right)dx  \\
			& \leq  (1-c{M}^{2/n} \delta)\int_{V}G\left(\nabla w,w^*\circ \nabla w\right)\det D^2 w \cdot dx \\
			& = (1-c{M}^{2/n} \delta) \int_{\nabla w( V )}G\left(p,w^* \right)dp \\
			& \leq (1-c{M}^{2/n} \delta)(1+C\delta) \int_{\nabla w( V )}G\left(p,u^*\right)dp \\
			& \leq (1-c{M}^{2/n} \delta) \int_{\nabla w( V )}G\left(p,u^*\right)dp \\
			& \leq (1-c{M}^{2/n} \delta) \int_{\nabla u( V )}G\left(p,u^*\right)dp < \infty.
		\end{align*}
		\endgroup
		Thus,  $ V =\emptyset$, and \eqref{eq:u geq w U} is proved.
	\end{proof}

	\begin{proof}[Proof of Theorem \ref{thm:smp 3}]
		Let $\Omega_0= \{x \in \Omega:\;  u=v\}$ denote the coincidence set.  Assume on the contradiction that $\Omega_0 \neq \emptyset$ and $  \Omega_0  \subsetneq \Omega$, then there is a point $X \in \partial \Omega_0 \cap \Omega $.
		Define
		\[\bar t(x)=  \sup\{ t :\; S_t^{v}(x) \subset \subset \Omega \setminus \Omega_0\}  \text{ for each point } x \in \Omega \setminus \Omega_0.\]
		
		We claim that there exists a point  $Y\in \Omega \setminus \Omega_0 $ such that the section $S_{\bar{t}(Y)}^v(Y) \subset \subset \Omega$.  Once we find such a point $Y$, we then assume that $\partial S_{\bar t(Y)}^{v}(Y)$ touches $  \Omega_0$ at point $0 \in \Omega$, and the rest of the proof is the same as that of Theorems \ref{thm:smp 1} and \ref{thm:smp 3}.

		We choose points $X_{i} \in \Omega \setminus \Omega_0$ such that $\lim_{i \to \infty} X_{i}=X$, and let $E_i= S_{\bar{t}(X_{i})}^v(X_{i})$. If our claim is false, then we would have $\partial E_i \cap \partial \Omega \neq \emptyset$. Then,
		\[\lim_{i \to \infty} \operatorname{diam}(E_i) \geq \lim_{i \to \infty} \operatorname{dist}(X_i,\partial\Omega)=\operatorname{dist}(X,\partial\Omega)>0. \]
		By taking a subsequence, we may assume that $ t_0=\lim_{i \to \infty} \bar{t}(X_{i})\in [0,\infty)$.	Since $v$ is $C^1$ and strictly convex, $E_i \to S_{t_0}^v(X)$  locally in $\Omega$ in the Hausdorff sense as $i \to \infty$.
		%		It is easy to check that $S_{t_0-\epsilon}^v(X) \subset E_i \subset S_{t_0+\epsilon}^v(X)$.
		Therefore, $ \operatorname{diam}(S_{t_0}^v(X)) > 0 $, which implies
		\[t_0 >0 \text{ and then }\operatorname{dist}(X,\partial S_{t_0}^v(X)) >0.\]
		Thus, $X \in E_i$ provided that $i$ is large enough, which is a contradiction to the fact that $ E_i   = S_{\bar{t}(X_{i})}^v(X_{i})  \subset \Omega \setminus \Omega_0$.
	\end{proof}
	
	\begin{Remark}
		If we replace the Lipschitz property of $F(x,r)$  with respect to $r$ by the non-increasing property with respect to $r$, then Theorems \ref{thm:smp 1} and \ref{thm:smp 3} still hold.
	\end{Remark}
	
	\begin{Remark}
		Using the Legendre transform, one can also obtain a strong maximum principle for $C^1$ strictly convex solutions of
		\[	\det D^2 u(x)=\frac{1}{g\left(\nabla u(x)\right)}\cdot\frac{ F\left(x, u\right) }{G\left(\nabla u, u^* \circ \nabla u\right)}.\]
		 Nevertheless,    whether or not a strong maximum principle holds for solutions of $\det D^2 u(x)= \frac{f(x)}{g\left( \nabla u(x) \right)} $, where  $f$ and $g$ are only bounded and positive,  remains 
an open problem.
	\end{Remark}

	\section{Applications to the Minkowski Problem}\label{chp:4}

	We first introduce some preliminaries in convex geometry, which can be found in   \cite{[P], [S]}. We also refer to \cite{[BLY], [HLY]} for  discussions on the Gauss image problem and the $L_p$ dual Minkowski problem, respectively.
	
	 We use $\K^{n+1}$ to  denote the set of convex bodies $K^{n+1}$ (compact, convex subsets with nonempty interior $\mathring{K}^{n+1}$) in $\R^{n+1}$,   $\K_{(o)}^{n+1}$ to denote  the set of convex bodies $K^{n+1}$ such that $0 \in \mathring{K}^{n+1}$, and   $\mathbb{S}^{n}$ to denote the unit sphere in $\R^{n+1}$. Let $K \in \mathcal{K}_{(o)}^{n+1}$ be a convex body. The {\sl radial function $\rho_K(t\theta ) $ of $K$ } is a homogeneous function of degree $-1$ on $\R^{n+1}$, given by
	\[  \rho_K(t\theta ) = t^{-1}\max \{ s \in \R :\; s\theta  \in K \} \text{ for } \theta  \in  \mathbb{S}^{n} \text{ and } t> 0.\]
	The {\sl normal cone of $K$ at $x \in \partial K$ } is the set-valued map
	\[N_K ( x)=\left\{\xi \in  \mathbb{S}^{n} :\; (y-x) \cdot \xi  \leq 0 \text { for all } y \in K\right\}.\]
	%	If  $N_K ( x)$ is a single vector, then  $x$ is a  regular point.
	The {\sl Gauss image map } is
	\[\boldsymbol{\alpha}_K(A)=\bigcup_{x \in r_{K}(A)} N_K ( x)\subset  \mathbb{S}^{n} \text{ for } A \subset  \mathbb{S}^{n},\]
where $ r_{K}(A)=\{  \rho_K(\theta )\theta:\; \theta \in A\},$ and
	  $\boldsymbol{\alpha}_K \left(\left\{\theta \right\}\right)$ will be  abbreviated as $\boldsymbol{\alpha}_K (\theta )$.
	By \cite[Lemma 3.1]{[BLY]}, $\boldsymbol{\alpha}_K$ maps closed sets into closed sets.
	
	The support function of $K$ is a one-homogeneous function on $\R^{n+1}$  given by
	\[	
	h_K(t\xi )=t\sup _{x \in K}\langle x, \xi  \rangle \text{ for }  \xi  \in \mathbb{S}^{n} \text{ and } t\geq 0.
	\]
	If we further assume that $ \partial K$ is $C^1$ and strictly convex, then  $\boldsymbol{\alpha}_K$ is one-to-one. By writing $\xi =\boldsymbol{\alpha}_K  ( \theta )$, we have
	\[\xi =\frac{\rho_K(\theta ) \theta -{\nabla}_{\mathbb{S}^{n}} \rho_K(\theta )}{\sqrt{\rho_K^2(\theta )+|\nabla_{\mathbb{S}^{n}} \rho_K(\theta )|^2}}, \quad \theta =\frac{{\nabla}_{\mathbb{S}^{n}} h_K(\xi )+h_K(\xi ) \xi }{\sqrt{|{\nabla}_{\mathbb{S}^{n}} h_K(\xi )|^2+h_K^2(\xi )}}  \]
	and
	\[ h_K(\xi )=\rho_K(\theta )\langle \theta , \xi \rangle, \quad \rho_K(\theta )=|{\nabla}_{\mathbb{S}^{n}}  h_K(\xi )+h_K(\xi ) \xi| , \]
	where  $\nabla_{\mathbb{S}^{n}} $  denotes the covariant derivative with respect to the standard Riemannian metric on $\mathbb{S}^n$.

	Notice that both $\rho_K$ and $h_K$ are representation functions of $K$ and are uniquely determined by $K$. The Gauss image map has duality.  In fact, let
	\[K^*=\left\{y \in \R^{n+1} :\; x \cdot y \leq 1 \text{ for all }   x \in K\right\}\]
	be the polar body of $K$. Then,
	\[ h_K (\xi )= \frac{1}{\rho_{K^*}(\xi )} \text{ and } \rho_K (\theta )= \frac{1}{h_{K^*}(\theta )}. \]
	%	and
	%	\[ \boldsymbol{\alpha}_K  \circ \boldsymbol{\alpha}_{K^*} \circ \boldsymbol{\alpha}_K = \boldsymbol{\alpha}_K  \text{ and } \boldsymbol{\alpha}_{K^*} \circ \boldsymbol{\alpha}_K  \circ \boldsymbol{\alpha}_{K^*}  = \boldsymbol{\alpha}_{K^*}  . \]
	It was shown in \cite{[HLY]} that the Gauss image map $\boldsymbol{\alpha}_{K^*}  $ of ${K^*} $ is the inverse Gauss image map $\boldsymbol{\alpha}_{K}^*$ of $K$, where
	\[\boldsymbol{\alpha}_{K}^*(\xi )=\{ \theta \in \mathbb{S}^n:\; \xi  \in \boldsymbol{\alpha}_{K}(\theta ) \}.  \]
	For simplicity, we write  $\rho_{K}^* , h_{K}^*, \boldsymbol{\alpha}_{K}^* $ for $ \rho_{K^*}, h_{K^*},\boldsymbol{\alpha}_{K^*} $, and omit the sub-index $K$ or $K^*$ when   no confusion happens.
	
	Formally, at  $\rho(\theta)\cdot \theta\in \partial K $, the Gauss curvature of $\partial K $ is,
	\[\K= \frac{1}{\det\left({\nabla}_{\mathbb{S}^{n}}^2h(\xi )+h(\xi ) \I\right)} \text{ where } \xi=\alpha_K(\theta),\]
	and the Jacobian of the Gauss image map is
	\[ \frac{d \boldsymbol{\alpha}_K(\omega)}{d\omega }= \frac{ \rho^{n+1}}{   h } \K = \frac{\left(h^2 +|{\nabla}_{\mathbb{S}^{n}}  h |^2\right)^{\frac{n+1}{2}}}{   h\det\left({\nabla}_{\mathbb{S}^{n}}^2h +h \I\right)},  \]
	where $d\omega$  is the standard volume form on $\mathbb{S}^n$.
	Similarly,  the Gauss curvature of $\partial K^*$ is
	%$\lambda (\boldsymbol{\alpha}_K(\cdot ))  $ is the measure $ \lambda(K, \cdot )$ defined by the reserve Gauss map,
	%\[ \int_{ \mathbb{S}^{n}} \Psi(\theta ) d \lambda(K, \theta )=\int_{ \mathbb{S}^{n}} \Psi\left(\alpha_K^*(\xi )\right) d \lambda(\xi ) \]
	%for each bounded Borel $\Psi:  \mathbb{S}^{n} \rightarrow \mathbb{R}$.
	 \[
	\K^*= \frac{1}{\det\left({\nabla}_{\mathbb{S}^{n}}^2h^*(\theta )+h^*(\theta ) \I\right)},\]
	 and the Jacobian of the inverse Gauss image map is
	\[\frac{d \boldsymbol{\alpha}_K^*(\omega)}{d\omega }= \frac{d \boldsymbol{\alpha}_{K^*}(\omega)}{d\omega }= \frac{ (\rho^{*})^{n+1}}{   h^* } \K^*  =\frac{\left({h^{*}}^2 +|{\nabla}_{\mathbb{S}^{n}}   h^{*} |^2\right)^{\frac{n+1}{2}} }
	{  h^{*}\det\left({\nabla}_{\mathbb{S}^{n}}^2h^{*} +h^{*}  \I\right)}  .\]
	At the corresponding points of the Gauss image map, we have
	\[ \frac{ \rho^{n+1}}{   h } \K  \cdot  \frac{ (\rho^{*})^{n+1}}{   h^* } \K^* =
	\frac{d \boldsymbol{\alpha}_K(\omega )}{d\omega  } \cdot \frac{d \boldsymbol{\alpha}_K^*(\omega)}{d\omega } =1, \]
	which implies that
	\begin{equation}\label{eq:gauss image equation 1}
		\frac{d \boldsymbol{\alpha}_K{^*}(\omega)}{d\omega }= \frac{h\det\left({\nabla}_{\mathbb{S}^{n}}^2h +h \I\right)}{ \left(h^2 +|{\nabla}_{\mathbb{S}^{n}}  h |^2\right)^{\frac{n+1}{2}}   }
		\text{ and }
		\frac{d \boldsymbol{\alpha}_K(\omega)}{d\omega }=\frac{h^{*}\det\left({\nabla}_{\mathbb{S}^{n}}^2h^{*} +h^{*}  \I\right)}
		{\left({h^{*}}^2 +|{\nabla}_{\mathbb{S}^{n}}   h^{*} |^2\right)^{\frac{n+1}{2}}   }  .
	\end{equation}

		 B{\"o}r{\"o}czky, Lutwak,  Yang,  Zhang and Zhao \cite{[BLY]} studied the Gauss image problem. They gave the necessary and sufficient conditions on submeasures $\nu$ and $\mu $ on $ \mathbb{S}^{n}$ such that there exists a convex body $K \in \mathcal{K}_{(o)}^{n+1}$  satisfying $\nu \left(\boldsymbol{\alpha}_K(\cdot )\right) = \mu$, where $\boldsymbol{\alpha}_K(\cdot )$.  Suppose now $\nu$ and $\mu$ on $\mathbb{S}^n$ are absolutely continuous Borel measures with densities $g$ and $f$, respectively, the  Gauss image problem is
	reduced to the following Monge-Amp\`ere equation
	\[    g\left(\frac{{\nabla}_{\mathbb{S}^{n}}h^*+h^*(\theta ) \theta }{\sqrt{|\nabla_{\mathbb{S}^{n}} h^*|^2+{h^*}^2}}\right)
	\frac{h^*\det\left({\nabla}_{\mathbb{S}^{n}}^2h^*+h^*\I \right)  }
	{\left( {h^*}^2 +|{\nabla}_{\mathbb{S}^{n}} h^* |^2\right)^{\frac{n+1}{2}} }  =f(\theta  )  \text{ for } \theta \in \mathbb{S}^{n} ,\]
	or equivalently
	\begin{equation}\label{eq:gauss image map}
		 g\left(\boldsymbol{\alpha}_K(\theta )\right) \frac{d \boldsymbol{\alpha}_K(\omega)}{d\omega} =f(\theta ) \text{ for } \theta  \in \mathbb{S}^n
	\end{equation}
	in the sense of measure (Aleksandrov).
	They demonstrated the comparison principle \cite[Lemma 3.8]{[BLY]} for the Gauss image problem and proved the uniqueness of the convex body up to dilation. We also refer to \cite{[CWX]} for smooth solutions to the Gauss image problem.
	
	 Here and below, we assume that the boundary of the convex body is $C^1$ and strictly convex for convenience. Clearly, the Gauss image problem \eqref{eq:gauss image map} is equivalent to
	\[
	\int_{{\alpha}_K(E)} g d \omega = \int_E f d\omega \text{ for each Borel } E \subset \mathbb{S}^n.
	\]
	From now on,  whenever we write an equation or inequality involving the Jacobian of the Gauss image map or the inverse Gauss image map, we always assume it holds in the sense of measure (Aleksandrov) by regarding the terms on both sides as Borel measures.

	%	  We now have the strong maximal principle for the equation
	%	\[	G\left(\boldsymbol{\alpha}_{K}^*,  h_{K}^*\right)\frac{d \boldsymbol{\alpha}_{K}^*(\omega)}{d\omega } = F\left(\xi , \rho_{K}^*  \right) f(\xi ),  \]
	%	where we regard $h_{K}^*= h_{K}^* \circ \boldsymbol{\alpha}_{K}^*$ as a function of $\xi $.
	\begin{Theorem}[The strong maximum principle for the Gauss image problem]\label{thm:smp Gauss image}
		Suppose that  $K_1, K_2 \in \K_{(o)}^{n+1}$,   $K_1 \subset K_2 $,  their boundaries are  $C^1$ and strictly convex, and
		\[
		\frac{d \boldsymbol{\alpha}_{K_1}(\omega)}{d\omega } \leq   \tilde{f}(\theta )  \cdot \frac{\tilde{F}\left(\theta , \rho_{K_1}   \right)}{\tilde{G}\left(\boldsymbol{\alpha}_{K_1} ,  h_{K_1}  \circ \boldsymbol{\alpha}_{K_1} \right)}  \text{ in }\hat U\subset S^n
		\]
		and
		\[
		\frac{d \boldsymbol{\alpha}_{K_2}(\omega)}{d\omega } \geq    \tilde{f}(\theta ) \cdot \frac{\tilde{F}\left(\theta , \rho_{K_2}   \right)}{	\tilde{G}\left(\boldsymbol{\alpha}_{K_2} ,  h_{K_2} \circ\boldsymbol{\alpha}_{K_2}  \right) }  \text{ in }\hat U\subset S^n
		\]
		  in the sense of measure (Aleksandrov). Assume that $0<\lambda \leq \tilde{f} \leq \Lambda$, both $\tilde{F}$ and $\tilde{G}$ are positive and Lipschitz. If
		$\partial K_2$ touches $\partial K_1$ at a point $\rho_{K_2}(\theta _0)\theta _0$ with  $\theta _0 \in U$,  then $\rho_{K_2}  \equiv \rho_{K_1}$ on $\hat U$.
	\end{Theorem}
	\begin{proof}
		Note that
		\[\rho_{K_1} (\theta )= \frac{1}{h_{K_1^*}(\theta )} \text{ and }
		\rho_{K_2} (\theta )= \frac{1}{h_{K_2^*}(\theta )}. \]
		Since $K_1 \subset K_2 $, we have
		$h_{K_2^*}(\theta) \leq   h_{K_1^*} (\theta )$. The assumption that	$K_2$ touches $K_1$ at  $\rho_{K_2}(\theta _0)\theta _0$ implies that  $h_{K_2^*}(\theta _0)$ = $h_{K_1^*}(\theta _0)$.

		We refer to \cite{[P]} for discussions on one-homogeneous convex functions. See  \cite[Lemma 7.3]{[BF]}, \cite[Page 26]{[HZ]}, or \cite[Page 11]{[JW13]}   for detailed calculations.
		We write points in $\R^{n+1}$ as $(x,x_{n+1})$ and let $H=\{ x_{n+1}=1\}$ be a hyperplane in $\R^{n+1}$.
		Let $K \in \mathcal{K}_{(o)}^{n+1}$ and suppose that  $\partial K$ is $C^1$ and strictly convex. The support function of $K^*$ is
		\[h^{*}(x,x_{n+1})=r\cdot ( h^{*}\circ \theta )  ,\]
		where
		\[  r(x,x_{n+1})= \sqrt{ |x|^2+|x_{n+1}|^2} \text{ and }\theta  (x,x_{n+1})=  \frac{\left(x,x_{n+1}\right)}{\sqrt{ |x|^2+|x_{n+1}|^2}} .\]
		The restriction of $ h^{*}$ on $H$ is the positive, strictly convex function
		\[  u(x) :=r\cdot \left(h^{*}\circ \theta _H(x)\right)=\frac{\sqrt{ |x|^2+1} }{\rho  \circ  \theta _H}  \text{ for } \theta _H(x)= \frac{\left(x,1\right)}{\sqrt{ |x|^2+1}}. \]
		The Gauss image map $\boldsymbol{\alpha}_{K}  $ is described by the gradient $\nabla  u$ of $u$ with respect to $x$,
		\[     \boldsymbol{\alpha}_{K}  \circ \theta _H =\left(\nabla u , \sqrt{1-|\nabla u|^2}\right)\]
		and
		\[ \det D^2 u= \frac{\det\left( \left({\nabla}_{\mathbb{S}^{n}}^2h^{*}\right)\circ \theta _H +h^{*}\circ \theta _H \right) }{(|x|^2+1)^{\frac{n+2}{2}}}. \]
		Moreover, let $u^*$ denote the Legendre transformation of $u$ with respect to $x$, then
		\[ h  \circ \boldsymbol{\alpha}_{K}  \circ \theta _H= \frac{1}{\left(|\nabla  u|^2+|u^* \circ \nabla  u|^2\right)^{\frac{1}{2}}} .\]
		
		We now regard $\theta _H,  \rho $ as functions of $x,u $ and regard $ \boldsymbol{\alpha}_{K}, h  \circ \boldsymbol{\alpha}_{K}$ as functions of $\nabla u, u^* \circ\nabla u $, respectively. Recalling \eqref{eq:gauss image equation 1}, the equation
		\[	\frac{d \boldsymbol{\alpha}_{K}(\omega)}{d\omega } \geq (\leq)  \tilde{f}(\theta ) \cdot \frac{\tilde{F}\left(\theta , \rho    \right)}{	\tilde{G}\left(\boldsymbol{\alpha}_{K} ,  h  \circ\boldsymbol{\alpha}_{K}  \right) }  \text{ for } \theta \in \mathbb{S}^n  \]
		of the convex body $K$ implies  the following Monge-Amp\`ere equation of $u$,
		\[
		\det D^2 u(x)\geq (\leq) \tilde{f} \circ  \theta _H (x) \cdot \frac{F\left( x,   u \right)}{	G\left( \nabla u, u^* \circ\nabla u  \right)} \text{ for } x\in \R^n
		\]
		in the sense of measure (Aleksandrov); where both $G$ and $F$ are positive and Lipschitz and satisfy (1.5) by the compactness of $S^n$.  Therefore, Theorem \ref{thm:smp Gauss image} follows directly from   Theorem \ref{thm:smp 3}.
	\end{proof}
	
	%	\begin{proof}
		%		%	We just recall \cite{[Mo2]}, the $1-$homogeneous functions says that
		%		%	$(2)$
		%		%	$$
		%		%	v(x)=\nabla v(x) \cdot x
		%		%	$$
		%		%	Here and below we let $r=|x|$ and we denote points in $\mathbb{S}^{n-1}$ by $\omega$. Writing
		%		%	$$
		%		%	v=r g(\omega)
		%		%	$$
		%		%	and choosing a coordinate system where $\omega$ is the last direction, we have
		%		%	$$
		%		%	D^2 v(\omega)=\left[\begin{array}{cc}
			%			%		\nabla_{\mathbb{S}^{n-1}}^2 g+g I_{n-1 \times n-1} & 0 \\
			%			%		0 & 0
			%			%	\end{array}\right]
		%		%	$$
		%	
		%		
		%%		Note that
		%%		\[\boldsymbol{\alpha}_{K}^*
		%%		= \frac{\rho^*(\xi ) \xi -{\nabla}_{\mathbb{S}^{n}} \rho^*}{\sqrt{|\rho^*|^2+|\nabla \rho^*|^2}}  \text{ and } h_{K}^*=\rho_{K}^* \cdot \langle \boldsymbol{\alpha}_{K}^*, \xi   \rangle ,  \]
		%		
		%	\end{proof}
	%	

	The $L_p$ dual Minkowski problem corresponds to the  Monge-Amp\`ere equation  \eqref{eq:lp dual Minkowski problem}. Recalling  \eqref{eq:gauss image equation 1}, we can rewrite \eqref{eq:lp dual Minkowski problem} as
	\begin{equation}\label{eq:Lp dual Minkowski equation}
		\frac{d \boldsymbol{\alpha}_K^*(\omega)}{d\omega }= \frac{1}{g \circ \boldsymbol{\alpha}_K^*}  \cdot  \frac{  {h^*}^{q}}{  {\rho^*}^{p}   }f(\xi)
		\text{ or }
		\frac{d \boldsymbol{\alpha}_K(\omega)}{d\omega }= \frac{1}{f \circ \boldsymbol{\alpha}_K}  \cdot  \frac{  {\rho}^{q}}{  {h}^{p}   }g(\theta),\; \theta \in \mathbb{S}^n .
	\end{equation}
	 Hence,  Theorem \ref{thm:Lp DMP 0}  is equivalent to
	
	\begin{Theorem}\label{lem:Lp dual Minkowski problem}
		Assume that both $f$ and $g$ are positive and bounded functions on $\mathbb{S}^n$,  either $f$ or $g $ is Lipschitz. Suppose $K_1, K_2\in \K_{(o)}^{n+1}$, and their support functions $h_{1}$ and $ h_{2}$  are  Aleksandrov solutions to the $L_p$ dual Minkowski problem \eqref{eq:Lp dual Minkowski equation}.
		%	 And one of $h_1, h_2$ is positive.
		If $p>q $, then  $h_{1}\equiv h_{2}$.  If $p=q $, then $h_{1} \equiv ch_{2}$ for an arbitrary constant $c>0$.
	\end{Theorem}
	
	\begin{proof}
		Recalling Caffarelli's interior regularity theory for Monge-Amp\`ere equations,  we see that $\partial K$ is  $C^1$ and strictly convex
provided that  $\partial K$ has positive bounded Gauss curvature.  Therefore, our assumptions imply that both $\partial K_1$ and $\partial K_2$ are $C^1$ and strictly convex.
		
		If $h_{1} \not\equiv h_{2}$, we may assume that $h_{1} >h_{2} $ somewhere on $ \mathbb{S}^{n}$. By enlarging  $K_2$, we can replace $h_{2}$ with $h_{3}=th_{2}$ for a suitable $t\geq 1$ and assume that
		\[ h_{3} \geq h_{1}  \text { on }  \mathbb{S}^{n}  \text { and } h_{3}=h_{1} \text { at some point } X_0 \in  \mathbb{S}^{n}.\]	
		Note that  $g \circ \boldsymbol{\alpha}_{K_{3}^*}=g \circ \boldsymbol{\alpha}_{K_1^*}$ and  $  \frac{  {h_{3}^*}^{q}}{  {\rho_{3}^*}^{p}   }=t^{p-q}  \frac{  {h_{1}^*}^{q}}{  {\rho_{1}^*}^{p}   }  \geq    \frac{  {h_{1}^*}^{q}}{  {\rho_{1}^*}^{p}   } $, where $\rho_i^*$ denotes the radial function of $K_i^*$, $i=1,2,3$. If $g$ is Lipschitz, applying  Theorem \ref{thm:smp Gauss image} to the dual bodies $K_{1}^*$ and $K_{3}^*$, we obtain that $h_{1}^* \equiv h_{3}^*$. If   $f$ is Lipschitz, applying  Theorem \ref{thm:smp Gauss image} to $K_1$ and $K_3$,  we obtain $h_{1} \equiv h_{3}$.
	\end{proof}
	
	\begin{Remark}
		In general, the $C^0$ estimate for convex bodies, i.e., whether or not they contain the origin, is one of the main difficulties in the study of convex geometry.
	\end{Remark}
	%\begin{Remark}
	%	A similar result should hold for  problems of the type
	%	\[ 	g ( \frac{\nabla h+h(\xi ) \xi }{\sqrt{|\nabla h|^2+h^2}})\det\left(\nabla_{i j}h(\xi )+h(\xi ) \delta_{i j}\right)=f(\xi ) h^{p-1}(\xi )\|\nabla h(\xi )+h(\xi ) \xi  \|_{\Q}^{n+1-q}. \]
	%\end{Remark}
	
	%Recall the $ \mathrm{NCH}$-condition for measures.
	%\begin{Definition}[$\mathrm{NCH}$-condition]\label{def:NCH}
	%	A measure $\mu $ on $\mathbb{S}^{n}$ is $ \mathrm{NCH}$ provided that it is a non-zero finite Borel measure that is not concentrated on a closed hemisphere.
	%\end{Definition}

	\section{The Strong Maximum Principle for Polynomial Measure}\label{chp:5}

First, we introduce an {\sl EDT property} for non-negative functions and measures, which is closed under finite sums and multiplication by bounded positive functions.
	
	\begin{Definition}[EDT property]\label{def:EDT  condition}
		Let $\mu=fdx$ be an absolutely continuous measure on $\R^{n}$. We say $f$ (or $\mu$) satisfies the EDT property if under any affine coordinates of $\R^n$ (which is still marked as  $(x_1,\cdots,x_n)$ for convenience), we can find positive integer $l$, nonnegative constant $b$, positive constants $a,c_0,\tau_0, M_0$, $S_1, \cdots, S_l$, and vectors $Z_1=(z_1,1), \cdots, Z_l=(z_l,1) \in \R^n$ such that 
 for any   $ \tau (0,  \tau_0 ) $ and any  $M > M_0$, the transformations
		\[  \T_{j, M,\tau}x=S_j\tau x+  ({M}-S_j)\tau x_n Z_j \text{ for } x=(x',x_n) \in \R^n,\quad j=1,\cdots,l,\]
	 satisfy that
		\[\sum_{j=1}^l \left|{f}\left( \T_{j, M,\tau}x\right)\right| \geq a\tau^{b}\left|{f}\left(x \right)\right| , \; \forall \; x  \in \R^n\; \text{and}\;  |x| \leq \frac{c_0}{{M}\tau } .  \]
	\end{Definition}

	\begin{Lemma}\label{lem:polynomial Lemma 1}
		Suppose $|\P(x)|$ is  a polynomial on $\R^{n}$. Then, $|\P(x)|$  satisfies the EDT property.
	\end{Lemma}
	
	\begin{proof}
		In this proof, we use $\alpha, \beta$ denote multi-index, i.e., $n$-tuples of nonnegative integers. Fix an affine coordinate, let $k_P$ be the minimum order of ${\P}$ at origin, and write
		\[  {\P}(x)=\sum_{|\beta| = k_P }  {P}_{\beta}x^{\beta}+\sum_{|\beta| \geq  k_P+1 } {P}_{\beta}x^{\beta}.\]
		We write $x=(ty,t)$ for $y \in \R^{n-1}$ and $t \in \R$, and consider the function ${\Q}(y)[t]= {\P}(ty,t)={\P}(x)$.  Then it maybe rewritten as
		\begin{equation} \label{poly express  Q 1}
			{\Q}(y)[t]=\sum_{k\geq 0 }t^{k_P+k}\sum_{|\beta| =  k_P+k} {\Q}_{\beta}(y)= \sum_{k\geq 0 }t^{k_P+k} {\Q}_{k}(y).
		\end{equation}
		
		Consider the sets of multi-index:     $J_{-1} :=\emptyset$,
		\[J_0 := \left\{ \alpha :\;\text{there exists a } \; \beta \geq  \alpha   \text{  such that the coefficient of } y^{\beta}   \text{ in } {\Q}_{0}(y) \text{ is non-zero } \right\} \]
		and
		\[J_j := \left\{  \alpha :\; \text{there exists a }\; \beta \geq  \alpha   \text{ such that the coefficient of } y^{\beta}   \text{ in } {\Q}_{j}(y) \text{ is non-zero }  \right\} .\]
		
		 For a constant $S$,  let
		\[   {\Q}_{0}(Sy+z)=\sum_{\alpha \in J_0} {\mathcal{R}}_{\alpha}(S,z)y^{\alpha}, \; y\in \R^{n-1},  \]
		where ${\mathcal{R}}_{\alpha} (S,z)=S^{|\alpha|}\sum_{\beta \in J_0, \beta \geq \alpha  } E_{\alpha,\beta}z^{\beta -\alpha}$ for   constants $E_{\alpha,\beta} \neq 0$.
		
		We claim that the polynomials ${\mathcal{R}}_{\alpha} $ defined  on $ \left(S,z\right) \in \R^+ \times \R^{n-1}$ are linearly independent, i.e., the image of $ \{{\mathcal{R}}_{\alpha}(S,z)\}_{\alpha \in J_0}$ in $\R^{|J_0|}$ is not contained in any hyperplane through $0$.  Suppose that there are constants $\{C_{\alpha}\}_{\alpha \in J_0}$ such that
		\[\sum_{\alpha \in J_0} C_{\alpha}{\mathcal{R}}_{\alpha} = 0 \text{ on } \R^+ \times \R^{n-1}.\]
		Then,  $\sum_{\alpha \in J_0} C_{\alpha}{\mathcal{R}}_{\alpha}  \equiv 0$ and for any nonnegative inters $k$,
		\begin{equation} \label{poly express 1}
			\sum_{\alpha \in J_0, |\alpha|=k}\sum_{\beta \in J_0, \beta \geq \alpha   } C_{\alpha}  E_{\alpha,\beta} z^{\beta -\alpha} \equiv 0  .
		\end{equation}
		Recalling the lexicographic order of monomials, we first compare the exponents of $z_1$ in monomials, and in the case of equality, we compare the exponents of $z_2$,  and so on. The matrix $\{E_{\alpha,\beta}\}$ is then expressed as an upper triangular matrix of full rank.
		%		\[\begin{pmatrix}
			%			H_{\alpha_1,\beta_1}, \cdots ,* & *, \cdots ,* & *, \cdots ,* & *, \cdots ,* \\
			%			0, \cdots ,0 & H_{\alpha_2,\beta_2},\cdots ,* & *, \cdots ,* & *, \cdots ,* \\
			%			\vdots & \vdots & \ddots & \vdots \\
			%			0, \cdots ,0 & 0, \cdots ,0 & 0, \cdots ,0 & H_{\alpha_s,\beta_s}, \cdots ,*
			%		\end{pmatrix} . \]
		Thus,  $\{C_{\alpha}\}_{\alpha \in J_0}$  are zero, and the claim is proved.
		Hence,  one can find positive constants $  S_i $ and points $ z_i  \in   \R^{n-1}$ such that the linear combinations of finitely many vectors $\{{\mathcal{R}}_{\alpha}(S_i,z_i)\}_{\alpha \in J_0}$ will give all unit elements in $\R^{|J_0|}$.  Therefore,   we can find a finite set of $(a_i, S_i,z_i) \in \R \times \R^+ \times \R^{n-1}$ such that
		\[\sum_{i} \left |a_i{\Q}_{0}(S_iy+z_i) \right |  \geq \sum_{\alpha \in J_0} |y^{\alpha} |.\]
		
		 Generally, when $k \geq 1$,  we write
		\[   {\Q}_{k}(Sy+z)=\sum_{\alpha \in J_k} {\mathcal{R}}_{\alpha,k}(S,z)y^{\alpha},\]
		and consider the polynomial
		\[ \tilde{\Q} (Sy+z)[t] = \sum_{k\geq 0} t^{k_P+k} \sum_{\alpha \in J_k \setminus \cup_{m=-1}^{k-1}  J_{m}}  {\mathcal{R}}_{\alpha,k}(S,z)y^{\alpha}.\]
		Similarly, we have that   $\{{\mathcal{R}}_{\alpha,k}(S,z)\}_{k \geq 0,\ \alpha  \in J_k \setminus \cup_{m=-1}^{k-1}  J_{m}}$ are linearly independent.
		%		 will yield all unit element $e_{\beta}=\{\delta_{\alpha \beta}\}_{\alpha \in J_k \setminus \cup_{m=0}^{k-1}  J_{m}}, \beta \in J_k \setminus \cup_{m=0}^{k-1}  J_{m}$ in $\R^{|J_k \setminus \cup_{m=0}^{k-1}  J_{m}|}$, where  $\delta_{\alpha \beta}$ is the Kronecker delta.
		And there are  $(a_i, S_i,z_i)  \in \R \times \R^+ \times \R^{n-1}, 1\leq i \leq l$, such that
		\[ \sum_{i=1}^l \left| a_i\tilde{\Q} (S_iy+z_i)[t]\right| \geq \sum_{k\geq 0}| t|^{k_P+k} \left(\sum_{\alpha \in J_k \setminus \cup_{m=-1}^{k-1}  J_{m}}| y^{\alpha} | \right).\]
		We can also assume that $|a_i|=1$ for simplicity. If $|t|$ is small,  then we have
		\[ \sum_{i=1}^l  \left| \tilde{\Q} (S_iy+z_i)[t]\right| \geq {\D}(y)[t]:=\frac{1}{2}\sum_{ k \geq 0}|t|^{k_P+k}\sum_{\beta  \in J_k  }  |y^{\beta}|,\]
		and
		\[  \begin{split}
			\sum_{i=1}^l \left| {\Q}(S_iy+z_i)[t]\right|
			&\geq {\D}(y)[t]+ \sum_{i=1}^l \left| {\Q}(S_iy+z_i)[t]-  \tilde{\Q} (S_iy+z_i)[t]\right|\\
			& \geq {\D}(y)[t] -C_{\Q} \sum_{ k \geq 1}|t|^{k_P+k}\sum_{\beta \in  \cup_{m=-1}^{k-1}  J_{m} }  |y^{\beta}| \\
			&\geq {\D}(y)[t] -C_{\Q}|t| {\D}(y)[t]\\
			& \geq \frac{1}{2}{\D}(y)[t].
		\end{split}
		\]
		Note that
		\[|(y+z)^{\beta} -y ^{\beta} |  = |\sum_{ \alpha<\beta} c_{ \alpha,\beta} z^{\beta-\alpha}y^{\alpha}|\leq C_2|z| (1+ |y^{\beta}|) \text{ for } |z| \leq C_1. \]
		When  $\kappa \geq 0$ is small, we have
		\[  \begin{split}
			\sum_{i=1}^l |{\Q} \left(z_i+\kappa S_i (y- z_i)\right)[t] |
			&\geq \sum_{i=1}^l |{\Q} (z_i+ \kappa S_i y)[t]|-C_{\Q} \kappa  {\D}(\kappa y)[t] \\
			& \geq \frac{1}{2} {\D}(\kappa y)[t] -C_{\Q} \kappa  {\D}(\kappa y)[t]   \\
			&\geq  \frac{1}{4} {\D}(\kappa y)[t].
		\end{split}
		\]
		Substituting  $x=(x',x_n)= x_n(y,1)$, $\kappa =M^{-1}$,  $t=M\tau x_n$ and $Z_i=(z_i,1), i=1,\cdots,l$,  we obtain
		\begingroup
		\allowdisplaybreaks  \begin{align*}
			\sum_{i=1}^l \left|{\P}\left( S_i\tau x+  ({M}-S_i)\tau x_n Z_i \right)\right|
			%			&= \sum_{i=1}^l \left|{\P}\left(S_i\tau \left(x',x_n\right)+  ({M}-S_i)\tau x_n \left(z_i,1\right) \right)\right| \\
			&= \sum_{i=1}^l \left|{\P}\left(S_i\tau x_n \left(y ,1\right)+  ({M}-S_i)\tau x_n \left(z_i,1\right) \right)\right| \\
			& =\sum_{i=1}^l \left|{\P}\left({M}\tau x_n\left(z_i+{M}^{-1} S_i (y- z_i) ,  1\right) \right)\right| \\
			& =\sum_{i=1}^l \left|{\Q}\left(z_i+{M}^{-1} S_i (y- z_i) \right)[{M}\tau x_n]\right| \\
			&\geq  \frac{1}{4} {\D}(M^{-1} y)[{M}\tau  x_n] \\
			& = \frac{1}{8} \sum_{ k \geq 0}|{M}\tau x_n|^{k_P+k}\sum_{\beta  \in J_k  }  |({M}^{-1} y)^{\beta}| \\
			&\geq  c_P \tau^{\deg {\P}}\left|{\P}\left(x \right)\right|
		\end{align*}
		\endgroup
		provided that $M$ is large and ${M}\tau |x_n|$ is small.
	\end{proof}
	
It follows directly from Definition \ref{def:EDT  condition} and Lemma \ref{lem:polynomial Lemma 1} that
	\begin{Corollary}\label{lem:polynomial Lemma 2}
		Property \eqref{eq:polynomial measure 4} implies the EDT property of  $f$.
	\end{Corollary}

	\begin{proof}[Proof of Theorem \ref{thm:smp dege 1}]
		By Lemma \ref{def:condition mu} and Corollary \ref{lem:polynomial Lemma 2}, we find that $\mu =fdx$ satisfies  the doubling condition and $f$ satisfies the EDT property.  Theorem \ref{thm:smp dege 1} follows directly from Theorem \ref{thm:smp dege 35} below.
	\end{proof}
	
	\begin{Theorem}\label{thm:smp dege 35}
		Suppose that $u  \geq v $, both $u$ and $v$ are generalized solutions to \eqref{monge ampere measure}.   Assume that the assumptions \eqref{eq:polynomial measure 1} hold. If $f$ satisfies the EDT property and $\mu =fdx$ satisfies  the doubling condition, then on each connected component $\Omega_i$ of $\Omega \setminus (\Sigma_u \cap \Sigma_v)$, we have either $u\equiv v$ or $u >v$.
	\end{Theorem}

	\begin{proof}
		Let $\Omega_0= \{x \in \Omega:\;  u=v\}$ denote the coincidence set. We only need to show that $\left(\partial \Omega_0 \cap \Omega \right)  \subset \Sigma_{u}\cap \Sigma_v  $. Assume on the contradiction that $ \left(\partial \Omega_0 \cap \Omega \right)  \not \subset \Sigma_v$.   As in the proof of Theorem \ref{thm:smp 0} in Section \ref{chp:3},  we can  assume that $0\in \Omega_0 \setminus \Sigma_v$, $ u(0)=v(0)=0$, $\nabla u(0)=\nabla v(0)=0$,  $u$ and $v $ are non-negative, and
		\[\{ x:\; v(x) < c_1x_n\}  \subset \Omega_0^c \cap \{ x_n > 0\} .\]
		Denote  $U=\{ x:\; v(x) < c_2x_n\}$ and let
		\begin{equation}\label{eq:sigma 1}
			\sigma(s ) :=\inf\left\{ 	u(x) -v(x)  :\; x\in \overline U,   x_n \geq  s \  \right\} >0 \text{ for } s>0,
		\end{equation}
		where $c_2 \in (0,c_1)$ is a small constant to be fixed later. We see that both $u$ and $v$ are $C^1$ and strictly convex around $0$,  $\partial U$ is strictly convex, and  $\lim_{s\to 0} \sigma(s)=0$.
		%  Choose small constant $\tau >0 $  such that $2C\tau^a  \leq  \sigma$ is small,
		%  and let
		%  \begin{equation} \label {8.8}
			%  	\Phi(x',x_n)= 4\Lambda \Phi_{a,b}(\tau x',x_n/ \tau^{n-1} ).
			%  \end{equation}
		%  Then
		%  \[\det D^{2}\Phi \geq4 {\Lambda},\]
		%  and
		%  \[ \Phi(x',x_n)   \leq  \frac{1}{2} \sigma|x'|^{a} + {M}|x_n|^b  \]
		%  for ${M}={M}(a, \sigma )$ is large enough. Note that $b >1 $,
		
		By the definition of the EDT property, there are positive integer $l$,   constant $b\geq 0 $, positive constants $a,c_0,\tau_0, M_0, S_1, \cdots, S_l$, and vectors $Z_1=(z_1,1), \cdots, Z_l=(z_l,1) $ in $\R^n$ such that
		\begin{equation}\label{eq:polynomial measure trans}
			\sum_{j=1}^l \left|{f}\left( \T_{j, M,\tau}x\right)\right| \geq a\tau^{b}\left|{f}\left(x \right)\right| \text{ for }  \T_{j, M,\tau}x=S_j\tau x+  ({M}-S_j)\tau x_n Z_j
		\end{equation}
		provided that $\tau\leq \tau_0 $ is small, $M > M_0$ is large and $|x| \leq \frac{c_0}{{M}\tau } $.
		
		Let
		\[\tau = \min\left\{ \tau_0, \frac{1}{1+4\sum_{j=1}  |  S_j |}\right\} \text{ and }M  \geq M_0+ \left[4\Lambda \sum_{j=1}^l\left(S_j+ S_j^{2-2n}\tau^{-2n}\right) + 1+ \frac{\tau^{-b}}{a}\right]^2.\]
		Let $c_2$ be small enough such that  $ U\subset\{ x:\; M\tau |x| \leq c_0  \}$. Consider the functions
		\[\Phi_j(x) =  v \left(  \T_{j, M,\tau} x\right) \text{ and } \Phi(x)= \sum_{j=1}^l	\Phi_j (x).\]
		 It follows from \eqref{eq:polynomial measure trans}   that
		\[\begin{split}
			\det  D^2 \Phi & \geq     \sum_{j=1}^l  \det D^2	\Phi_j (x)	 \\
			& \geq  \sum_{j=1}^l  (M-S_j)^2 S_j^{2n-2}\tau^{2n} f \left(  \T_{j, M,\tau} x\right) \\
			& \geq    M^{3/2}  \sum_{j=1}^l f \left(  \T_{j, M,\tau} x\right) \\
			&\geq   M^{3/2} a\tau^b   f( x) \\
			& \geq  M f(x) \text{ in } U.
		\end{split}\]
		
		By the convexity and $C^1$ regularity of $v $ at $0$, we have that
		\begingroup
		\allowdisplaybreaks
		\begin{align*}
			\Phi(x)= \sum_{j=1}^{l}	\Phi_j (x)
			& =  \sum_{j=1}^{l} v\left( S_j\tau x+  ({M}-S_j)\tau x_n Z_j \right) \\
			&\leq \sum_{j=1}^{l} \left[S_j\tau \cdot v(x)+(1-S_j\tau )\cdot v\left(\frac{({M}-S_j)\tau}{  (1-S_j\tau )} x_n Z_j\right)\right]\\
			& \leq \sum_{j=1}^{l}  \left[S_j\tau \cdot v(x) + o(|x_n|)\right]\\
			%& \leq \left(\sum_{j=1}^{l} 2 S_j\tau\right) \cdot v(x) + o(|x_n|)\\
			& \leq  \frac{1}{2} v(x) + o(|x_n|)  \text{ on } \partial U\cap \{    x_n \leq  s_0 \}
		\end{align*}
		\endgroup
		and then
		\[u(x)  \geq  \frac{1}{2}u(x)+\frac{c_2}{2} x_n   \geq   \Phi(x)+\frac{c_2}{4} x_n  \text{ on } \partial U\cap \{    x_n \leq  s_0 \}\]
		provided that $s_0=s_0(c_2) >0 $ is small enough.
		
		%		\[ \sum_{j=1}^{l} \left|{\Q}_i'\left(Z_j'+\epsilon (y'-Z_j)\right)\right| \geq \epsilon^{\deg {\Q}_i}\left|{\Q}_i'\left(y'\right)\right|.  \]
		%		

		Let $\delta \leq C\sigma(s_2)$ be a small positive constant, and consider the function
		\[ w := \frac{ v+ \delta \Phi+ \frac{c_2\delta}{4}x_n}{1+\delta} .\]
		If $c_2$ is small enough, the same discussion as the proof  of Theorem \ref{thm:smp 1} gives $ u \geq  w$ in $ U$.  Thus, $ \left(\nabla u(0) -\nabla w(0)\right) \cdot e_n \geq 0$, which is  a contradiction to
 the fact that $\nabla u(0)= 0$ and $\nabla w(0)= \frac{c_2\delta}{4} e_n$. The proof is completed.
	\end{proof}

		\section{The Set of Non-Strictly Convex Points}\label{chp:6}

		%		\begin{Theorem}[\cite{[Mo1]}, Theorem 1.1]\label{thm:mooneys Theorem 1.1}
			%			Suppose that $u$ is convex, then $u$ is strictly convex away from a singular set $\Sigma$ with $\mathcal{H}^{n-1}(\Sigma)=0$. Therefore, $\mathcal{H}^{n-1}(\Sigma_u )=0$,  and the complement of the closed set $\Sigma_u$ is connected provided that $\Sigma_u$ is closed.
			%		\end{Theorem}
		%			Suppose that $u$ is convex and  $\det D^2 u \geq \eta$, where $\eta \geq 0$ is continuous. Theorem \ref{thm:mooneys Theorem 1.1} implies  $H^{n-1}(\Sigma_u\setminus Z_{\eta})=0$.

		%	 Our strategy is to control $ \Sigma_u \setminus \L_{\eta}$ from $\Sigma_u \setminus Z_{\eta}$.   By Lemma \ref{lem:Pogorelov Line},  $\Sigma_u$ consists of convex sets. This might be an uncountable union of convex sets. So, we recall  Mooney's delicate estimation for sections at non-strictly convex points, namely Lemmas 3.2 and 3.3 in  \cite{[Mo1]}.   Our proof is in the spirit of Mooney's proof of $H^{n-1}(\Sigma_u)=0$ when $\det D^2 u \geq \lambda >0$. We are going to show that for any $\epsilon>0$ and point $x\in  \Sigma_u \setminus \L_{\eta}$, there is a sequence $r_{k} \rightarrow 0$ such that
		%		\[	\M w\left(B_{r_{k}}(x_0)\right)>\frac{1}{\epsilon} r_{k}^{n-1} \text{ for } w= u+|x|^2.
		%		\]
			In	\cite[Mooney]{[Mo1]}, it is shown that if $u$ is convex and
		$ \det D^2u \geq \lambda > 0$, then $u$ is strictly convex away from a singular set $\Sigma$ and $\mathcal{H}^{n-1}(\Sigma )=0$.  It should be noted that the definition of strict convexity in \cite{[Mo1]} is usually weaker than the classical definition \eqref{eq:def:strict convex} of strict convexity. However, when $u \in C_{loc}^1$ or $ \M u$ satisfies the doubling condition, they are equivalent. Regardless of which definition is used, with some modifications to the proof details following the spirit of \cite{[Mo1]}, we still have $\mathcal{H}^{n-1}(\Sigma_u )=0$, and basing on this, we can make some generalizations to the degenerate case.

		\begin{Lemma}[{\cite[Section 3]{[Mo1]}}]\label{lem:section n}
			For each point ${x_0}$, if there are positive constants $r_{x_0}$ and $c_{x_0}$ such that $\det D^2 u \geq c_{x_0} >0 $ in $B_{r_{x_0}}({x_0}) \subset \Omega$.  Then
 			for  any subgradient   $ p \in \partial  {u}({x_0})$,
			\[\left|S_{t, p}^{u}({x_0}) \right| \leq  C_{{x_0}} t^{\frac{n}{2}} \text{ for  all  small } t >0   .\]
		\end{Lemma}
		\begin{proof}
			%	Suppose that $x \in \mathcal{Z}_{\eta}^c$. Let $a  = \frac{1}{2}\operatorname{dist}(x,\mathcal{Z}_{\eta})$ and assume that  $\eta \geq c_x >0 $ in $B_{a}(x)$.
			By Lemma \ref{lem:size section 3},  we have for any $t >0$ and  any subgradient  $ p \in \partial  {u}({x_0})$,
			\[
			c_x\left|S_{t, p}^{u}({x_0}) \cap B_{r_{x_0}}({x_0})\right|^2 \leq C t^n    .
			\]
			Observing  that $S_{t, p}^{u}({x_0}) $ is convex and $x_0 \in S_{t, p}^{u}({x_0})$, we have
			\[ \left|S_{t, p}^{u}({x_0}) \right|  \leq \left[\frac{ \operatorname{diam}(\Omega)}{r_{x_0}}\right]^n \cdot \left|S_{t, p}^{u}({x_0}) \cap B_{r_{x_0}}({x_0})\right| \leq C_{{x_0}} t^{\frac{n}{2}}.  \]
		\end{proof}
		
		\begin{Definition}
			For a convex function $u$ in $\Omega$,  it follows from Caffarelli [5] that $\Sigma_{u}$ is the union of segments $L $ on which ${u}$ is linear.  Let $\tilde{\Sigma}_{u}$ denote the union of open segments ${L}$ on which ${u}$ is linear.
		\end{Definition}
		
		\begin{Lemma}[{\cite[Section 3]{[Mo1]}}]\label{lem:section n n+1}
			Assume that  ${x_0} \in \tilde{\Sigma}_{u}$, $ p \in \partial  {u}({x_0})$, and
			\[ \left|S_{t, p}^{u}({x_0}) \right|   \leq C_{{x_0},p} t^{\frac{n}{2}} \text{ for any small } t.  \]
			Then,
			\[ \left|S_{t, p}^{w}({x_0}) \right|   \leq \tilde{C}_{{x_0},p} t^{\frac{n+1}{2}} \text{ for any small } t,  \]
			where $w={u}+|x|^2$.
		\end{Lemma}
		\begin{proof}
			We may assume that ${x_0}=0 $ and $ {u}(0)=0$. Since $0 \in \tilde{\Sigma}_{u}$,  there exists a line segment $L \subset \Omega$ and a subgradient $p_0\in \partial {u}(0)$ such that
			\[0 \in \mathring{L}  \subset \left\{x :\;   {u}(x)= p_0 \cdot x \right\} .\]

			Now, for any fixed subgradient $ p \in \partial {u}(0) $, we still have $L \subset S_{t, p }^{u}(0)$. Let $P_L^{\bot} S_{t, p }^{u}(0)$ denote the projection of $S_{t, p }^{u}(0)$ onto the $(n-1)$-plane which is perpendicular to    $L$. By convexity, we have
			\[ |P_L^{\bot} S_{t, p }^{u}(0)|  \leq C\frac{\left|S_{t, p}^{u}({0}) \right|}{\operatorname{dist}(0,\partial L)}\leq \frac{C_{{x_0},p}}{\operatorname{dist}(0,\partial L)}  t^{\frac{n}{2}}. \]
			Note that since $w={u}+|x|^{2}$, again by the convexity of $u$ we have
			\[S_{t, p  }^{w}({0})\subset S_{t, p  }^{|x|^{2}}({0}) \subset t^{\frac{1}{2}} B_1(0).\]
			Therefore,
			\[|S_{t, p  }^{w}({x_0}) | \leq  \frac{C_{{x_0},p}}{\operatorname{dist}(0,\partial L)} t^{\frac{n}{2}} \cdot t^{\frac{1}{2}} \leq  \tilde{C}_{{x_0},p}t^{\frac{n+1}{2}} .\]
		\end{proof}
		
		\begin{Lemma}[\cite{[Mo1]}, Lemma 3.2]\label{lem:mooneys Lemma 3.2} Assume that $w={u}+|x|^{2}$ for a convex function ${u}$. Fix a point $x_0$, if for all subgradient  $ p \in \partial  w(x_0) $,
			\begin{equation}\label{eq:mo singular pointwise}
				\left|S_{t, p }^{w}(x_0)\right| \leq C_{x,p}t^{\frac{n+1}{2}}
			\end{equation}
			holds for all small $t>0$, then, for any fixed  subgradient  $ p \in \partial  w(x_0) $, letting
			\[
			d_1(t) \geq d_2(t) \geq \cdots \geq d_n(t)
			\]
			denote the axis lengths of the John ellipsoid of the section $S_{t, p}^u(x_0)$, we have
			\begin{equation}\label{eq:mo singular pointwise 1}
				\frac{d_{n-1}(t)}{t^{1 / 2}} \rightarrow 0   \text { as } t \rightarrow 0
			\end{equation}
		\end{Lemma}

		In the assumption of Lemma \ref{lem:mooneys Lemma 3.3} below, comparing with the original statement,  we only need \eqref{eq:mo singular pointwise} and \eqref{eq:mo singular pointwise 1} to be satisfied for a certain subgradient $p$ instead of for all subgradients, which is easy to check from the original proof.
		\begin{Lemma}[\cite{[Mo1]}, Lemma 3.3]\label{lem:mooneys Lemma 3.3}
			Assume that $w=u+|x|^{2}$ for a convex function $u$. Suppose there exists a subgradient  $ p \in \partial  u(x_0) $ satisfying \eqref{eq:mo singular pointwise} and \eqref{eq:mo singular pointwise 1}. Then for any $\epsilon>0$, there is a sequence $r_{k} \rightarrow 0$ such that
			\begin{equation}\label{eq:mo singular pointwise 2}
				\M w\left(B_{r_{k}}(x_0)\right)>\frac{1}{\epsilon} r_{k}^{n-1} .
			\end{equation}
		\end{Lemma}

		Now we review Definition \ref{def:ZL sets} for $\L_{\eta}$.
		\begin{Lemma} \label{lem:section sigma tilde}
			Suppose that $u $ is convex and $\det D^2 u \geq \eta $ in $\Omega$, where $\eta \geq 0$ is continuous. Let $w=u+|x|^2$. Then for any $x_0 \in  \Sigma_u \setminus \L_{\eta}$, there exists a subgradient $p \in \partial w(x_0)$ such that    \eqref{eq:mo singular pointwise} and \eqref{eq:mo singular pointwise 1} holds for all small $t>0$. Furthermore, if $\M u $ additionally satisfies the doubling condition, then we can replace the assumption $x_0 \in  \Sigma_u \setminus \L_{\eta}$ with $x_0 \in  \Sigma_u \setminus \tilde{\L}_{\eta}$.
		\end{Lemma}
		
		\begin{proof}
			%		Therefore, Lemma \ref{lem:mooneys Lemma 3.2} (equations \eqref{eq:mo singular pointwise}-\eqref{eq:mo singular pointwise 1}) holds for all $ X \in  \Sigma_u \setminus \mathcal{Z}_{\eta}$ and  $ p \in \partial  w(X) $ for all subgradient  $ p \in \partial  w(X) $.	
			Let $x_0 \in \Sigma_u\setminus \L_{\eta} $,  there is  a supporting hyperplane function $\ell(x)= p_{x_0}\cdot (x-x_0) +u(x_0)$ of $u$ at $x_0$ and a line segment $L \subset \Omega$   such that
			\[  L  \subset  \left\{x :\;   u(x)= \ell(x)\right\} \text{ and } x_0 \subsetneq L.\]
			By the definition of $x_0 $ and $\L_{\eta}$, we can assume that $L  \not \subset \mathcal{Z}_{\eta} $.
			Since $\mathcal{Z}_{\eta}  \cap L$ is a closed subset of $L$,   we can take a point $X \in  \mathring{L} \setminus \mathcal{Z}_{\eta}$. Then,
			\[X \in \tilde{\Sigma}_u \setminus \mathcal{Z}_{\eta}.\]
			For simplicity, we may assume that
			\[ X=0,\ p_{x_0}=0, \ \operatorname{dist} \left(0, \partial L \right) = \kappa >0,\]
			and write  $x_0= |x_0| e$ for a unit vector $e$. Applying Lemmas \ref{lem:section n} and \ref{lem:section n n+1} to the point $X=0$, we obtain that
			\[
			|S_{t, 0}^{w}(x_0)|=|S_{t, 0}^{u}(0)+\{x_0\}| =|S_{t, 0}^{u}(0)|\leq  C_{0,0}t^{\frac{n+1}{2}} .
			\]
			This gives \eqref{eq:mo singular pointwise} for $x_0$ and $0 \in \partial u(x_0)$.
			
			For each fixed $t>0$, consider the John ellipsoid of $S_{t, 0}^{w}(0)$, with unit directions $e_1(t) , e_2(t) , \cdots , e_n(t) $, and with
 the corresponding  axis lengths $d_1(t) \geq d_2(t) \geq \cdots \geq d_n(t)$.
			Since $0 \in    \tilde{\Sigma}_u \setminus \mathcal{Z}_{\eta}$, combining Lemmas  \ref{lem:section n}, \ref{lem:section n n+1} and \ref{lem:mooneys Lemma 3.2}, we have	
			\[\frac{d_{n-1}(t)}{t^{1 / 2}} \rightarrow 0  \text{ as } t \rightarrow 0.\]
			Note that
			\[ S_{t, 0}^{w}(0) \subset S_{t, 0}^{|x|^2}(0) \subset t^{\frac{1}{2}}B_{1}(0) .\]
			Therefore, for any small $\epsilon>0$,
			\[    S_{t, 0}^{w}(0) \subset  \left\{|x\cdot e_{n-1} | +|x\cdot e_{n} | \leq C d_{n-1} (t)\leq C \epsilon t^{\frac{1}{2}}  \right\} \cap t^{\frac{1}{2}}B_{1}(0)  \]
			provided that $t>0$ is small enough.
			%this is
			%\[    w(x) \geq  h \text{ on } \left\{|x\cdot e_{n-1} | +|x\cdot e_{n} | \leq C \epsilon t^{\frac{1}{2}}  \right\}\]
			%for all small $h$.
			
			Now suppose that $t>0$ is small. Use vectors $\tilde{e}_{n-1}(t), \tilde{e}_{n}(t)$ to denote the projections  of $ {e}_{n-1}(t), {e}_{n}(t)$ onto the hyperplane which is perpendicular to $e$. Let  $E$  be  the plane generated by  $\tilde{e}_{n-1}(t), \tilde{e}_{n}(t)  $, $\operatorname{P}_E$ be the projection map onto $E$, and
			\[T_E=\{x \in \R^n :\; |\operatorname{P}_E x  |  \leq C \epsilon t^{\frac{1}{2}}   \}. \]
			 Since
			\[L\cap t^{\frac{1}{2}}B_{1}(0) \subset S_{t, 0}^{w}(0) \subset t^{\frac{1}{2}}B_{1}(0),\]
			we have that
			\[   S_{t, 0}^{w}(0) \subset  2T_E \cap t^{\frac{1}{2}}B_{1}(0) \text{ if } t^{\frac{1}{2}} \leq \kappa. \]
			Then,
			\[  u (x) \geq w (x)  -|x|^2 \geq t-\frac{1}{4}t \geq \frac{1}{2}t  \text{ in }     \frac{1}{2}t^{\frac{1}{2}}B_{1}(0) \setminus 2T_E. \]
			Therefore,
			\[   S_{t/2, 0}^{u}(0)  \cap  \frac{1}{2} t^{\frac{1}{2}}B_{1}(0) \subset  2T_E.\]
			%Thus,
			%\[  u (x) \geq w (x)  -\frac{1}{2}|x|^2 \geq \frac{1}{2}h  \text{ for } x  \in    t^{\frac{1}{2}}B_{1}(0) \cap 2 F. \]
			
			The  assumption $\operatorname{dist} \left(0, \partial L \right) = \kappa  >0$ implies that $-\kappa e\in S_{t/2, 0}^{u}(0) $.
			By convexity, $S_{t/2, 0}^{u}(0) \cap t^{\frac{1}{2}}B_{1}(0) $ is inside the convex cone generated from  $S_{t/2, 0}^{u}(0)  \cap  t^{\frac{1}{2}}B_{1}(0)$  with vertex at  $-\kappa x_0$, by  a 
calculation we obtain that
			%
			%\[  u (x_0) \geq   \frac{1}{2}h  \text{ for }  x  \in    \kappa^{-1}  \left(t^{\frac{1}{2}}B_{1}(x_0) \cap 2  V\right) \]
			%and
			%\[  w (x_0) \geq   \frac{1}{2}h  \text{ for }  x  \in    \kappa^{-1}   \left(t^{\frac{1}{2}}B_{1}(x_0) \cap 2  V\right) . \]
			%Thus,
			%\[    S_{t, 0}^{w}(x_0) \subset 2C\kappa^{-1}  \epsilon t^{\frac{1}{2}}   \left\{|x\cdot e_{n-1} | +|x\cdot e_{n} | \leq 4C\kappa^{-1}  \epsilon t^{\frac{1}{2}}  \right\}\]
			%\[    S_{t/2, 0}^{w}(x_0) \subset   4C\kappa^{-1}  \epsilon t^{\frac{1}{2}}  F\]
			\[  S_{t/2, 0}^{w}(x_0) \subset   S_{t/2, 0}^{u}(x_0) \cap   t^{\frac{1}{2}}B_{1}(x_0) \subset  4C\kappa^{-1}  \epsilon t^{\frac{1}{2}} T_E\cap   t^{\frac{1}{2}}B_{1}(x_0).\]
			This gives \eqref{eq:mo singular pointwise 1}  for $x_0$ and $0 \in \partial u(x_0)$.
			
			If $\M u$ also satisfies the doubling condition. Recalling Lemma \ref{lem:Pogorelov Line}, for each point $x_0 \in \Sigma_u \setminus \tilde{\L}_{\eta}$,   we can still find a  segment $ L$ containing $x_0$ such that $L \not \subset  \mathcal{Z}_{\eta} $. The remaining proof is the same.
		\end{proof}

		\begin{proof}[Proof of Theorem \ref{thm:hausdorff n-1}]
			%	The proof is in the spirit of Mooney's proof. We first  show that for any $\epsilon>0$ and $x\in  \Sigma_u \setminus \L_{\eta}$, there is a sequence $r_{k} \rightarrow 0$ such that
			%	\[	\M w\left(B_{r_{k}}(x)\right)>\frac{1}{\epsilon} r_{k}^{n-1} .\]
			 The proof  of Theorem \ref{thm:hausdorff n-1}  now directly follows from \cite{[Mo1]}. For the completeness of the article, we restate it here.  For any $\epsilon>0$ and point $x\in  \Sigma_u \setminus \L_{\eta}$, we apply Lemmas \ref{lem:mooneys Lemma 3.3}  and \ref{lem:section sigma tilde}. There is a sequence $r_{k} \rightarrow 0$ such that
			\begin{equation}\label{eq:singular pointwise 4}
				\M w\left(B_{r_{k}}(x)\right)>\frac{1}{\epsilon} r_{k}^{n-1} .
			\end{equation}
			We now cover  $\Sigma_u \setminus \mathcal{Z}_{\eta} $ with balls $\bigcup_{x_j \in \Sigma_u \setminus \L_{\eta}} B_{r_j}(x_j)$ such that
			\[\M w\left(B_{r_{j}}(x_j)\right)>\frac{1}{\epsilon} r_{j}^{n-1} \text{ and } r_j < \epsilon .\]
			By the Vitali covering Lemma, there exists  a countable refinement family, disjoint balls $\{ B_{r_j}(x_j)\}_{j\in J}$ such that
			\[\bigcup_{j \in J} B_{5r_j}(x_j)  \supset \bigcup_{x_j \in \Sigma_u \setminus \L_{\eta}} B_{j} \supset  \Sigma_u \setminus   \L_{\eta} .\]
			  Note that
			\[\sum_{j \in J} \left|\operatorname{diam} B_{5r_j}(x_j)\right| ^{n-1}\leq C_n\sum_{j \in J}r_{j}^{n-1} \leq   C_n\epsilon\M w\left( \bigcup_{x_j \in J}  B_{r_{j}}(x_j)\right) \leq C_n \epsilon \M w(\Omega) .  \]
			One may assume that $ \M w(\Omega)  < \infty$ by exhausting $\Omega$ with subdomains. Letting $\epsilon \to 0^+ $,  we prove  tthat $ H^{n-1}(\Sigma_u \setminus \L_{\eta}) =0$.
			
			If $\M u$ satisfies the doubling condition, the same proof gives $  H^{n-1}(\Sigma_u \setminus \tilde{\L}_{\eta}) =0$.
		\end{proof}
		
		 The following is the $W^{2,1+\epsilon}$ regularity result inside internal sections.
		\begin{Theorem}[\cite{[Phi1]}, Theorem 2]\label{lem:psf Theorem 2}
			Suppose that $u$ is a convex solution to
			\[\det D^2 u =f \text{ in } S_{t}^u(x_0) \subset \subset \Omega\]
			with $f$ as in \eqref{eq:polynomial measure 4}. Then, $u \in W^{2,1+\epsilon}\left(S_{t/2}(x_0)\right)$ for some positive consatant $\epsilon$ depending only on $\lambda, \Lambda, n$, and the doubling constants of $fdx$.
		\end{Theorem}

		\begin{proof}[Proof of Theorem \ref{thm:w21 estimate}]
			Theorem \ref{lem:psf Theorem 2} gives the local $W^{2,1+\epsilon}$ regularity of $u$ on $\Omega \setminus \Sigma_u$. Due to the doubling condition of $fdx$, $\Sigma_u$ is a closed set.  
This leads us to prove that $|D^2u|$ cannot concentrate on $\Sigma_u$.
 By Theorem \ref{thm:hausdorff n-1} and our assumptions, we have
			\[\mathcal{H}^{n-1}(\Sigma_u  )=0.\]
			Then for any $\epsilon>0$, we can cover $\Sigma_u $ with open balls $\left\{ B_{r_j}(x_j)\right\}$ such that $\sum_{j=1}^{\infty}r_j^{n-1} < \epsilon$.  Hence,
			\[  \begin{split}
				\sum_{j=1}^{\infty}\int_{B_{r_j}(x_j) } |D^2 u | dx
				\leq  C(n)\sum_{j=1}^{\infty}\int_{B_{r_j}(x_j) } \Delta u dx
				& =  C(n)\sum_{i=1}^{\infty} \int_{\partial B_{r_i}} D_{\nu}u d s \\
				&\leq C ||u||_{Lip} \cdot \sum_{i=1}^{\infty} r_i^{n-1}< C\epsilon ||u||_{Lip}.
			\end{split}\]
			Therefore, the second derivatives of $u$ cannot concentrate on $\Sigma_u$. The proof is completed.
		\end{proof}
		%\begin{Remark}
		%	The main modification here relies on the fact that the information of level sets can pass along the line segment of $\Sigma_u$; this allows us to estimate the level sets with points within $\L_{\eta}$ from the outside.
		%\end{Remark}
		
			\appendix
		%\counterwithin*{equation}{section}
		\renewcommand\theequation{\thesection\arabic{equation}}
		
		\section{  Examples}\label{chp:a}
		
		%In this subsection,
		
	In the non-degenerate case, in Example \ref{exa:non-dege} and Remark \ref{rem:non-dege} below, we show there exist a family of convex functions with a common non-strictly convex point set $\Sigma$  such that all the coincidence sets are exactly $\Sigma$, and the Hausdorff dimension of $\Sigma$ is close to $n-1$. 
 For the degenerate case, in Example \ref{exa:dege}, we will construct a family of convex functions whose non-strictly convex point set $\Sigma$ is precisely the zero set of the Monge-Am\`ere measure, such that the coincidence set lies exactly on one side of their common non-strictly convex point set $\Sigma$.
	
		Caffarelli proves that the Pogorelov line cannot extend to the boundary point where $u$ is smooth enough, see \cite[Lemma 3.4]{[JX2]} for a proof.
		\begin{Lemma}[Caffarelli]\label{lem:strictly convexpogorelov}
			Suppose that $\Omega$ is convex, and $u\in C(\overline {\Omega})$ satisfies
			\[ \det D^2 u \geq  \lambda >0\text{ in } \Omega \text{ and } u=\varphi \text{ on } \partial \Omega.
			\]
		If both $\partial \Omega$ and $\varphi $ are pointwise $ C ^{1,\alpha}$ at $x_0 \in \partial \Omega$  for $\alpha > 1-\frac{2}{n}$, then $u$ is strictly convex at $x_0$ along any oblique direction.   Hence, if $u$ is linear on a segment $L \subset \Omega$, then $x_0 $ cannot be the endpoint of $L$.
		\end{Lemma}

		\begin{Example}\label{exa:non-dege}
			In the following steps $(1)-(5)$, we  recall the functions $v$ and $w$, and the set $S_{\delta}$ from  Mooney \cite[Section 4]{[Mo1]}.
			\begin{enumerate}
				\item Let $\epsilon>0$ be a small constant, $\gamma=1-2^{-3 \epsilon}$ and $	\delta=\frac{2 \epsilon}{1+3 \epsilon}$. We  construct a standard cantor set
				\[
				S_{\delta}=\left[-\frac{1}{2}, \frac{1}{2}\right]-\bigcup_{i, k} I_{i, k}
				\]
				as follows: for each $k \geq 0$, let $\left\{s_{i, k}\right\}_{i=1}^{2^{k}}$ denote the center of the remaining intervals, respectively; and we remove $2^k$ intervals $I_{i, k}$ centered at points $\left\{s_{i, k}\right\}_{i=1}^{2^{k}}$ with length $ \gamma 2^{-(1+3 \epsilon) (k-1)}$.
				
				\item Consider the function
				\[v(s)=\sum_{k=1}^{\infty} \sum_{i=1}^{2^{k-1}} 2^{-2(1+2 \epsilon) k} v_0\left(2 \gamma^{-1} 2^{(1+3 \epsilon) k}\left(s-s_{i, k}\right)\right),\]
				where
				\[v_0(s)=
				\begin{cases}
					|s|, & |s| \leq 1, \\
					2|s|-1, & |s|>1.
				\end{cases}\]
				For any point $s_0 \in S_{\delta}$,  $v$ separates from its supporting hyperplane function faster than $c_0|s-s_0|^{2-\delta}$ for some $c_0 >0$.
				
				\item Let $p,p'>1$ satisfy $\frac{1}{p}+\frac{1}{p'}=\frac{3}{2}$ with $p=2-\delta$. Let
				\[\Omega= \left(-1,1 \right ) \times  \left(-1,1 \right ) \times  \left(-1,1 \right )  \in \R^3 .\]
				There are positive constants $a$ and $b$, and a convex function of the type
				\[w\left(x_1, x_2, x_3\right)=\left(x_1^p+x_2^{p'}\right) f\left(x_3\right)  \text{ for } f(x_3)=a+bx_2^2\]
				such that
				\[ \det D^2 w \geq 1 \text{ in } \Omega.\]
				
				\item For each large $t>t_0$, we solve the following Dirichlet problem
				\[\operatorname{det} D^2 u(\cdot, t)=1   \text { in } \Omega, u(x_1,x_2,x_3, t) =v\left(x_1\right)+t \left|x_2\right|  \text { on } \partial \Omega.  \]
				We then use $w$ to construct a lower barrier function at each point on
				\[E_{\delta}= S_{\delta}  \times \{0\} \times  \left(-1,1 \right ) , \]
				and prove that $u(\cdot, t)$ is linear in the direction of $x_3$ on $ E_{\delta}$. Therefore,
				\[  E_{\delta} \subset \Sigma_{u(\cdot, t)}  \text{ for each } t \geq t_0.\]
				Applying Lemmas \ref{lem:Pogorelov Line} and \ref{lem:strictly convexpogorelov},  we find that
				\[ \Sigma_{u(\cdot, t)} \subset (-1,1)  \times \{0\} \times  (-1,1). \]
				\item If $n > 3$, we can consider the family of functions
				\[u \left(x_1, x_2, x_3,t\right)+x_4^2+\cdots+x_n^2.\]

				\item For each $t_2 > t_1 > t_0$, we have  $u(\cdot, t_1) \leq u(\cdot, t_2)$ in $\Omega$, and $u(\cdot, t_1) = u(\cdot, t_2)$ on $(-1,1)  \times \{0\} \times  \{\pm 1\}$.  By Lemma \ref{lem:Pogorelov Line}, we have
				\[    E_{\delta} \subset \Sigma_{u(\cdot, t_1)} \subset \Sigma_{u(\cdot, t_2)} \subset (-1,1)  \times \{0\} \times  (-1,1). \]
				Note that since $ u(\cdot,t_1) $ and $ u(\cdot,t_2) $ are different somewhere, we conclude that
				\[  E_{\delta} \subset \Sigma_{u(\cdot, t_1)} \cap \Sigma_{u(\cdot, t_2)} =\{x \in \Omega:\;  u(\cdot, t_1)=u(\cdot, t_2)\} \subsetneq \Omega  .\]
				
			\end{enumerate}
		\end{Example}

		\begin{Remark}\label{rem:non-dege}
			In addition to Example \ref{exa:non-dege}, consider the set
			\[ J_{\delta}= \bigcup_{i,k}\{s_{i, k},s_{i, k}-\gamma2^{-1-(1+3 \epsilon) k},s_{i, k}+\gamma2^{-1-(1+3 \epsilon) k}\}.\]
			Then, $J_{\delta}$ is discrete and all the limit points of $J_{\delta}$ are in $S_{\delta}$.   One  can check for all $u(\cdot, t)$ defined in Example \ref{exa:non-dege} the following
			\begin{equation}\label{eq:Mooney function}
				\Sigma_{u(\cdot, t)} =\left(S_{\delta} \cup J_{\delta}\right)  \times \{0\} \times  \left(-1,1 \right),
			\end{equation}
			by verifying steps (1)-(3) below:
			\begin{enumerate}
				\item For any point $s_0 \in  S_{\delta} \cup J_{\delta}$,  $v$ separates from its supporting hyperplane function faster than $c_0|s-s_0|^{2-\delta}$ for some $c_0 >0$.  Then, we have
				\[\left(S_{\delta} \cup J_{\delta}\right)  \times \{0\} \times  \left(-1,1 \right) \subset \Sigma_{u(\cdot, t)} . \]
				\item  Suppose that point
				\[Z=(z_1,z_2,z_3) \in    \left(S_{\delta} \cup J_{\delta}\right)  \times \{0\} \times  \left(-1,1 \right)  \text{ but } Z \notin \Sigma_{u(\cdot, t)}.\]
				By Lemmas \ref{lem:Pogorelov Line} and \ref{lem:strictly convexpogorelov}, $Z$ is in a segment $L$ on which $u$ is linear, and the endpoints of $L$ are on $(-1,1)  \times \{0\} \times  \{ \pm 1 \} $. Thus, $z_1 \notin S_{\delta} \cup J_{\delta}$.
				\item $S_{\delta} \cup J_{\delta}$ is a closed set. We may assume that $z_1 \in (s_1,s_2) \subset [-1,1]\setminus (S_{\delta} \cup J_{\delta})$ for points $s_1,s_2 \in S_{\delta} \cup J_{\delta}$. Consider the rectangle
				\[  E=(s_1,s_2) \times \{0\} \times  \left(-1,1 \right)\]
				Now, $u(\cdot, t)$ is a convex function, and $u$ is linear on the  $\{Z\} \cup \partial E$. Thus, $u$ is linear on $E$. This contradicts Lemma \ref{lem:Pogorelov Line}, which states that $\dim E < \frac{3}{2}$. In conclusion, we have proved  \eqref{eq:Mooney function}.
			\end{enumerate}
		\end{Remark}
		
		The next two examples concern degenerate cases.
		\begin{Example}\label{exa:2 SW}
			Sawyer and Wheeden \cite{[SW]} introduced the following convex function
			\[w(y, s)=y^{2}\zeta(s) \text{ for } (y,s) \in B_{c_0}(0) \subset \R^2,\]
			where $c_0$ is small, and  $\zeta$ solves the differential equation
			\[2\zeta\zeta^{\prime \prime}-4\zeta'^{2}=1, \zeta(0) >0, \zeta'(0)=0.\]
			In this example, we have
			\[ \det D^2 w = y^{2}.\]
			And $\Sigma_w$ is exactly on the $s$-axis.
			
			For each fixed $t>t_0$, where $t_0$ is large, the function
			\[w  (y,s,t) =  t^{-1}w(y,ts)=t^{-1}y^{2}\zeta(t s )\]
			satisfies
			\[ \det D^2 w (\cdot ,t) = y^{2}.\]
			Away from the $s$-axis, we have $w(\cdot ,t) > w(\cdot,1)$ somewhere. Thus, $w(\cdot ,t) $ only touches $w(\cdot,1)$ on the $s$-axis, where they have the value $0$.
			
			When $n >3$, we write points in $\R^n$ as $x=(x'',x_{n-1},x_n  )$.
			Let
			\[v(x)=w\left(x_{n-1}, x_n,1\right)+|x''|^2\]
			and
			\[ u(x)= \begin{cases}
				w\left( x_{n-1},x_n,t_0\right)+|x''|^2   & \text{ if }  x_{n} \geq 0,\\
				w\left( x_{n-1}, x_n,1\right)+|x''|^2  &
				\text{ if } x_{n} \leq 0.
			\end{cases}\]
			Then,  $\Sigma_u=\Sigma_v=\{ x_n= 0\}$, and the coincidence set is $\{ x_n \leq 0\}$.
		\end{Example}
		
		Inspired by Example \ref{exa:2 SW}, we can further construct function $u$ such that $\Sigma_u$ is flat along exactly one direction.
		\begin{Example}\label{exa:dege}
			Consider the function
			\[ w(x)=\frac{1}{2}\left[\left( x_{n}-\frac{1}{2}|x''|^2\right)^2\zeta(x_{n-1})+|x''|^2\right] \text{ for } \ x=\left(x'',x_{n-1},x_n\right) \in \R^n. \]
			Then we have
			\[ \det D^2 w =\left(\frac{1}{2} \zeta \zeta''-\zeta'^2\right)^2 \cdot  \left(1-\zeta W \right)^{n-2} W^2, \ \text{ where  }W (x) =x_{n}-\frac{1}{2}|x''|^2.\]
			Solving the equation
			\[\frac{1}{2} \zeta \zeta''-\zeta'^2=1 , \zeta(0)=1,  \zeta'(0)=0,\]
			we find that $w$ is convex  and
			\[ \det D^2 w\approx \left(x_{n}-\frac{1}{2}|x''|^2\right)^2 \text{ in } B_{c_0}(0)\]
			provided that $c_0>0 $ is small. Let $c_1>0$ be small, and let $\tilde{w}$ be the convex solution to
			\[ \det D^2 \tilde{w}= c_1 \left(x_{n}-\frac{1}{2}|x''|^2\right)^2 \text{ in } B_{c_0}(0) ,\quad \tilde{w} =w \text{ on } \partial B_{c_0}(0) .\]
			By comparison principle, $\tilde{w} \geq w$. Combining with the convexity of $\tilde{w}$, we see that
			\begin{equation}\label{eq:example v growth}
				\frac{1}{2}\left[c_2	\left( x_{n}-\frac{1}{2}|x''|^2\right)^2+|x''|^2\right]\leq \tilde{w}(x ) \leq \frac{1}{2}\left[C_2	\left( x_{n}-\frac{1}{2}|x''|^2\right)^2+|x''|^2\right]
			\end{equation}
			for some positive constants $c_2$ and $C_2$. Then,
			\[\tilde{w}=w=\frac{1}{2}|x'|^2 =  x_n \text{ on } E_0=\left\{x_n =\frac{1}{2}|x'|^2 \right\} \]

			We claim that
			\[ \Sigma_{\tilde{w}}= E_0 .\]
			Suppose to the contrary that there exists a point $X \in \Sigma_{\tilde{w}} \setminus E_0$. According to Lemma \ref{lem:Pogorelov Line}, $X$ is in a convex set $E$ on which $u$ is linear, and all extremal points of $E$ are on $E_0\cup  \partial B_{c_0}(0)  $.  Note that $ \tilde{w} \in C^{2}(\partial B_{c_0}(0))$ and $\det D^2 \tilde{w}>0$ outside $E_0$, applying Lemma \ref{lem:strictly convexpogorelov}, we find that all extremal points of $E$ are on $E_0$. Therefore, $\tilde{w}=x_n$ on $E$. By \eqref{eq:example v growth}, we have $X \in E \subset E_0$, which contradicts with $X \in  \Sigma_{\tilde{w}} \setminus  E_0$.

			For each large $t>t_0$, the function
			\[w (x,t) = t^{2} \tilde{w}( t^{-1}x'', t^{2} x_{n-1}, t^{-2}x_n) \]
			solves
			\[\det D^2  w(\cdot,t)=c_1 \left(x_{n}-\frac{1}{2}|x''|^2\right)^2,\]
			and
			\[ w(\cdot,t) \geq 	\frac{1}{2}\left[ct^{-2}	\left( x_{n}-\frac{1}{2}|x''|^2\right)^2+|x''|^2\right] \geq  \tilde{w}(x).   \]
			Thus, $w(\cdot,t) \geq v$, and $w(\cdot,t) $ only touches $v$ on $E_0$.   Let
			\[v(x)=w\left(x,1\right) \]
			and
			\[ u(x)= \begin{cases}
				w\left( x,t_0\right)   &  x_n  \geq \frac{1}{2}|x''|^2,\\
				w\left( x,1\right) & x_n  \leq \frac{1}{2}|x''|^2.
			\end{cases}\]
			Then,  $\Sigma_u=\Sigma_v=E_0$ while the coincidence set $\{x:\;u(x)=v(x)\}$ equals $\{ x_n \leq \frac{1}{2}|x''|^2\}$.
		\end{Example}

%		\begin{Remark}
%			Similarly, let $\zeta$ solves
%			\[\frac{\alpha}{1+\alpha} \zeta \zeta''-\zeta'^2=1,\ \zeta(0)=1,\ \zeta'(0)=0.\]
%			Denote $h=\zeta^{-\alpha}$; this gives the Emden–Fowler equation $h''=-\alpha h^{2\alpha +1}$. So we have a smooth convex solution $\zeta $ near the origin. Consider the function
%			\[w(x)=\frac{1}{2}\left[\left| x_{n}-\frac{1}{2}|x''|^2\right|^{1+\alpha}\zeta(x_{n-1})+|x''|^2\right], \ x=\left(x'',x_{n-1},x_n\right), \ \alpha >0.  \]
%			Now, $w$ is convex, and $\det D^2 w  \approx \left| x_{n}-\frac{1}{2}|x''|^2\right|^{2\alpha} $ around the origin.
%		\end{Remark}
%		

		\par\bigskip\noindent
	\textbf{Acknowledgments.}	The authors are
	grateful to Professor Tianling Jin for discussions and suggestions on the earlier version of this paper. We also thank Professors Jian Lu and Xu-Jia Wang for their helpful comments on this topic.

	\medskip
	
	\noindent \textbf{Conflict of interest:} All authors certify that there is no actual or potential conflict of interest about this article.

		\bigskip
		
		\noindent H. Jian
		
		\noindent Department of Mathematical Sciences, Tsinghua University\\
		Beijing 100084, China\\[1mm]
		Email: \textsf{hjian@tsinghua.edu.cn}

		\medskip
		
		\noindent X. Tu
		
		\noindent Department of Mathematics, The Hong Kong University of Science and Technology\\
		Clear Water Bay, Kowloon, Hong Kong\\[1mm]
		Email:  \textsf{maxstu@ust.hk}
		
	\end{document}